\documentclass[10pt, oneside]{amsart}

\usepackage{amsmath, amsmath, amsthm, amsfonts}
\usepackage{amssymb}
\usepackage{mathrsfs}
\usepackage{eucal}
\usepackage[all]{xy}
\usepackage{float}
\usepackage{stmaryrd}
\usepackage{hyperref}
\usepackage{todonotes}
\usepackage{enumitem}
\usepackage{eucal}
\usepackage[textwidth=14.5cm, textheight=22cm]{geometry}
\usepackage{setspace}
\usepackage{color}
\usepackage[T1]{fontenc}

\hypersetup{%
  bookmarksnumbered=true,%
  colorlinks=true,%
  linkcolor=blue,%
  citecolor=blue,%
  filecolor=blue,%
  menucolor=blue,%
  urlcolor=blue,%
  pdfnewwindow=true,%
  pdfstartview=FitBH}

\let\OLDthebibliography\thebibliography
\renewcommand\thebibliography[1]{
 \OLDthebibliography{#1}
 \setlength{\parskip}{1.4pt}
 \setlength{\itemsep}{0pt plus 0.2ex}
}

\newtheorem{theorem}{Theorem}[section]
\newtheorem{proposition}[theorem]{Proposition}
\newtheorem{corollary}[theorem]{Corollary}
\newtheorem{lemma}[theorem]{Lemma}
\newtheorem{notation}[theorem]{Notation}
\newtheorem*{thrm}{Theorem B}
\newtheorem*{thrA}{Theorem A}

\newtheorem{definition}[theorem]{Definition}
\newtheorem{example}[theorem]{Example}

\newtheorem{remark}[theorem]{Remark}

\newcommand{\Mdef}[2]{\newcommand{#1}{{#2}}}

\Mdef{\bA}{\mathbb{A}}
\Mdef{\bB}{\mathbb{B}}
\Mdef{\bC}{\mathbb{C}}
\Mdef{\bD}{\mathbb{D}}
\Mdef{\bE}{\mathbb{E}}
\Mdef{\bF}{\mathbb{F}}
\Mdef{\bG}{\mathbb{G}}
\Mdef{\bH}{\mathbb{H}}
\Mdef{\bI}{\mathbb{I}}
\Mdef{\bJ}{\mathbb{J}}
\Mdef{\bK}{\mathbb{K}}
\Mdef{\bL}{\mathbb{L}}
\Mdef{\bM}{\mathbb{M}}
\Mdef{\bN}{\mathbb{N}}
\Mdef{\bO}{\mathbb{O}}
\Mdef{\bP}{\mathbb{P}}
\Mdef{\bQ}{\mathbb{Q}}
\Mdef{\bR}{\mathbb{R}}
\Mdef{\bS}{\mathbb{S}}
\Mdef{\bT}{\mathbb{T}}
\Mdef{\bU}{\mathbb{U}}
\Mdef{\bV}{\mathbb{V}}
\Mdef{\bW}{\mathbb{W}}
\Mdef{\bX}{\mathbb{X}}
\Mdef{\bY}{\mathbb{Y}}
\Mdef{\bZ}{\mathbb{Z}}
\Mdef{\cL}{\mathcal{L}}
\Mdef{\cF}{\mathcal{F}}

\Mdef{\mcA}{\mathcal{A}}
\Mdef{\mcB}{\mathcal{B}}
\Mdef{\mcC}{\mathcal{C}}
\Mdef{\mcD}{\mathcal{D}} 
\Mdef{\mcE}{\mathcal{E}}
\Mdef{\mcF}{\mathcal{F}}
\Mdef{\mcG}{\mathcal{G}}
\Mdef{\mcH}{\mathcal{H}} 
\Mdef{\mcI}{\mathcal{I}}
\Mdef{\mcJ}{\mathcal{J}}
\Mdef{\mcK}{\mathcal{K}}
\Mdef{\mcL}{\mathcal{L}}

\Mdef{\mcM}{\mathcal{M}}
\Mdef{\mcN}{\mathcal{N}}
\Mdef{\mcO}{\mathcal{O}}
\Mdef{\mcP}{\mathcal{P}}
\Mdef{\mcQ}{\mathcal{Q}}
\Mdef{\mcR}{\mathcal{R}}
\Mdef{\mcS}{\mathcal{S}}
\Mdef{\mcT}{\mathcal{T}}
\Mdef{\mcU}{\mathcal{U}}
\Mdef{\mcV}{\mathcal{V}}
\Mdef{\mcW}{\mathcal{W}}
\Mdef{\mcX}{\mathcal{X}}
\Mdef{\mcY}{\mathcal{Y}}
\Mdef{\mcZ}{\mathcal{Z}}

\Mdef{\At}{\tilde{A}}
\Mdef{\Bt}{\tilde{B}}
\Mdef{\Ct}{\tilde{C}}
\Mdef{\Et}{\tilde{E}}
\Mdef{\Ht}{\tilde{H}}
\Mdef{\Kt}{\tilde{K}}
\Mdef{\Lt}{\tilde{L}}
\Mdef{\Mt}{\tilde{M}}
\Mdef{\Nt}{\tilde{N}}
\Mdef{\Pt}{\tilde{P}}

\newcommand{\pscr}{{ \mathscr{P} }}

\newcommand{\torus}{{ \mathbb{T} }}
\DeclareMathOperator{\ch}{Ch}
\DeclareMathOperator{\Sp}{Sp}

\def\endash{\mathchar"707B}

\newcommand{\cell}{\endash \textnormal{cell} \endash}

\newcommand{\dmod}{\textnormal{dmod}}
\newcommand{\leftmod}{\endash \textnormal{mod}}
\newcommand{\leftdmod}{\endash \textnormal{dmod}}
\newcommand{\leftmodin}{\endash \textnormal{mod} \endash}
\newcommand{\leftdmodin}{\endash \textnormal{dmod} \endash}

\newcommand{\Einfty}{\textnormal{E}_\infty}
\newcommand{\EGinfty}{\textnormal{E}^G_\infty}
\newcommand{\Eoneinfty}{\textnormal{E}^1_\infty}

\newcommand{\einf}{\textnormal{E}_\infty}

\newcommand{\co}{\colon \!}

\newcommand{\ho}{\textnormal{Ho}}

\newcommand{\smashprod}{\wedge}

\DeclareMathOperator{\colim}{colim}

\newcommand{\tensor}{\otimes}

\newcommand{\adjunction}[4]{
\xymatrix{
#1:#2 \ar@<0.7ex>[r] &
\ar@<0.7ex>[l] #3:#4
}}

\newcommand{\C}{\mathbb{C}}
\newcommand{\Q}{\mathbb{Q}}
\newcommand{\T}{\mathbb{T}}
\newcommand{\Z}{\mathbb{Z}}
\newcommand{\cA}{\mathcal{A}}
\newcommand{\cC}{\mathcal{C}}
\newcommand{\cE}{\mathcal{E}}

\newcommand{\cK}{\mathcal{K}}
\newcommand{\cM}{\mathcal{M}}
\newcommand{\cN}{\mathcal{N}}
\newcommand{\cP}{\mathcal{P}}
\newcommand{\cQ}{\mathcal{Q}}
\newcommand{\cV}{\mathcal{V}}

\newcommand{\gl}{\mathcal{Gl}}
\newcommand{\bbS}{\mathbb{S}}
\newcommand{\PcOC}{P\cOC}

\newcommand{\Hom}{\mathrm{Hom}}
\newcommand{\cOCFhat}{(\cOC)_{\cF}^{\wedge}}
\newcommand{\Iinj}{I\endash\mathrm{inj}}
\newcommand{\Jinj}{J\endash\mathrm{inj}}
\newcommand{\Icof}{I\endash\mathrm{cof}}

\newcommand{\piT}{\pi^{\T}}

\newcommand{\lra}{\longrightarrow}
\newcommand{\spO}{\textnormal{Sp}}
\newcommand{\TSP}{\mathbb{T}\textnormal{Sp}}
\newcommand{\spS}{\textnormal{Sp}^\Sigma}

\Mdef{\infl}{\mathrm{inf}}
\Mdef{\defl}{\mathrm{def}}
\Mdef{\res}{\mathrm{res}}
\Mdef{\ind}{\mathrm{ind}}
\Mdef{\coind}{\mathrm{coind}}
\Mdef{\Top}{\mathsf{Top}}
\Mdef{\sset}{\mathsf{sSet}}

\Mdef{\Comm}{\mathsf{Comm}}

\newcommand{\algin}[2]{#1 \endash \textnormal{alg-in}\endash #2}

\newcommand{\st}{\; |\;}

\newcommand{\lspec}{\cL(\bR^\infty)\endash\TSP_\bQ}

\newcommand{\adjunct}{{ \,\,\raisebox{-0.1\height}{$\overrightarrow{\longleftarrow}$}\,\, }}

\definecolor{darkgreen}{RGB}{85,107,47}

\newcommand{\GSp}{G \spO}

\newcommand{\Sn}{\Sigma_n}

\newcommand{\lla}{\leftarrow}
\newcommand{\sm}{\wedge}
\newcommand{\Ninfty}{N_{\infty}}

\newcommand{\cO}{\mathcal{O}}
\newcommand{\Pic}{\mathrm{Pic}}
\newcommand{\Cn}{C\langle n \rangle}
\newcommand{\cOC}{\cO_C}

\newcommand{\OmegaC}{\Omega^1_C}
\newcommand{\cOCnhat}{(\cOC)_{\Cn}^{\wedge}}
\newcommand{\cOCfork}{\cOC^{\lrcorner}}
\newcommand{\PcOCfork}{\PcOC^{\lrcorner}}

\newcommand{\cOCforkmod}{\cOCfork \endash \dmod}
\newcommand{\GcOCforkmod}{\Gamma\cOCfork \leftmod}
\newcommand{\PcOCforkmod}{\PcOCfork \endash \dmod}

\newcommand{\Rafork}{\bbSfork_a}
\newcommand{\Raforkmod}{\bbSfork_a \leftmod}
\newcommand{\bbSfork}{\bbS^{\lrcorner}}

\newcommand{\Rfork}{R^{\lrcorner}}

\newcommand{\siftyV}[1]{S^{\infty V(#1)}}

\newcommand{\cOcF}{\cO_{\cF}}
\newcommand{\cEi}{\cE^{-1}}
\newcommand{\cOCFh}{\cO_{C(\cF)}^{\wedge}} 
 
\newcommand{\efp}{E\cF_+}
\newcommand{\elr}[1]{E\langle #1 \rangle}
\newcommand{\jcell}{\textnormal{cell} \endash}
\newcommand{\fork}{\lrcorner}

\newcommand{\tF}{\cEi\cOcF}
\newcommand{\bbK}{\mathbb{K}}

\title[An algebraic model for rational na\"{i}ve-commutative ring $SO(2)$--spectra]{An algebraic model for rational na\"{i}ve-commutative ring $SO(2)$--spectra and equivariant elliptic cohomology}
\author[Barnes]{David Barnes}
\address[Barnes]{Pure Mathematics Research Centre, Queen's University Belfast, UK}
\email{d.barnes@qub.ac.uk}
\author[Greenlees]{J.P.C. Greenlees}
\address[Greenlees]{Warwick Mathematics Institute, Coventry, UK}
\email{john.greenlees@warwick.ac.uk}
\author[K\k{e}dziorek]{Magdalena K\k{e}dziorek}
\address[K\k{e}dziorek]{Mathematical Institute, Utrecht University, The Netherlands}
\email{m.kedziorek@uu.nl}

\begin{document}

\begin{abstract}
Equipping a non-equivariant topological $\Einfty$--operad with the trivial $G$--action gives
an operad in $G$--spaces. For a $G$-spectrum, being an algebra over
this operad does not provide any multiplicative norm maps on homotopy
groups. Algebras over this operad are called
na\"{i}ve--commutative ring $G$--spectra.
In this paper we take $G=SO(2)$ and we
show that commutative algebras in
the algebraic model for rational $SO(2)$--spectra model
rational na\"{i}ve--commutative ring $SO(2)$--spectra. In particular,
this applies to show that the  $SO(2)$-equivariant
cohomology associated to an elliptic curve $C$ of \cite{gre05} is represented by an
$\Einfty$--ring spectrum. Moreover, the category of modules 
over that $\Einfty$--ring spectrum is equivalent to
the derived category of sheaves over the elliptic curve $C$ with the
Zariski torsion point topology. 
\end{abstract}

\maketitle

\tableofcontents

\section{Introduction}

\subsection*{Rational equivariant cohomology theories}
We are interested in the category of rational \makebox{$G$--spectra}, where $G=SO(2)$ is the circle group and
the indexing universe 
is a complete $G$--universe $U$.
This is a model for the rational equivariant stable homotopy category
where all $G$--representation spheres are invertible.
Building on work of Greenlees and Shipley \cite{tnqcore} and Barnes \cite{barnesmonoidalmodelso2}, Barnes, Greenlees, K\c{e}dziorek  and Shipley \cite{BGKS} gave a symmetric monoidal algebraic model for the category of rational $G$--spectra, when $G=SO(2)$.
As a consequence, one obtains a model for rational ring $G$--spectra in terms of
monoids in the algebraic model. However, this does not imply analogous results
about strict commutative rational ring $G$--spectra ($\Comm$-algebras). This is because of
the well-known but surprising result that \emph{symmetric}
monoidal Quillen functors can fail to preserve \emph{commutative}
monoids (algebras for the operad $\Comm$) in the equivariant setting.

Recent work of Blumberg and Hill \cite{BlumbergHillNorms}
describes a class of commutative multiplicative structures on the
equivariant stable homotopy category. These multiplicative structures are governed by $G$--operads called $\Ninfty$--operads. Roughly speaking, such a multiplicative structure is characterised by the set of Hill-Hopkins-Ravenel norms (see \cite{HHR_annals}) which exist on the commutative algebras corresponding to it. Originally \cite{BlumbergHillNorms} considered $\Ninfty$--operads for a finite group $G$, however the definition can be extended to any compact Lie group $G$. The present paper uses only the most trivial $\Ninfty$--operad $\Eoneinfty$ which for any group $G$ is modelled by a non-equivariant topological $\Einfty$--operad inflated to the equivariant world of $G$-topological spaces.

The work of Blumberg and Hill raises a question: which level of commutative ring $\T$--spectra is modelled by commutative algebras in the algebraic model of \cite{BGKS}?

For $G$ a finite group, an understanding of which levels of commutativity are visible in the algebraic models for rational $G$---spectra was done in \cite{BGKfinite}. The new ingredient there was an analysis of which localised model structures $L_A\GSp$ can be 
right-lifted to the category of $\mcO$-algebras in $L_A\GSp$.
Recall that $L_A\GSp$ is a left Bousfield localisation at an object $A$ of the positive stable model structure on $G$-spectra and $\mcO$ is an $\Ninfty$--operad.

The work mentioned above is related to another surprising fact in equivariant homotopy theory, namely that a left Bousfield localisation of a genuinely commutative ring $G$--spectrum might fail to be commutative. Historically, the first example of that phenomenon appeared in \cite{McClureTate}, where McClure showed that for a finite group $G$ and a family $\cF$ of proper subgroups of $G$,  $\widetilde{E}\cF \simeq L_{\widetilde{E}\cF}\bS$ is not a strictly commutative ring spectrum. Here $\widetilde{E}\cF$ is the cofibre  of the natural map from a universal space $E\cF_+$ to $S^0$. Recently, more examples have been discussed in \cite{HillHopkins} and \cite{hill17}.

The failure of $\widetilde{E}\cF$ to be genuinely-commutative for a finite $G$ is related
to the fact that the restriction of $\widetilde{E}\cF$ to $H\Sp$ for any
proper subgroup $H$ of $G$ is trivial, while $\widetilde{E}\cF$ is not
contractible in $\GSp$. If $\widetilde{E}\cF$ were commutative then,  there would exist a ring map $*\simeq N_H^G\mathrm{res}_H^G \widetilde{E}\cF \lra \widetilde{E}\cF$, which gives a contradiction, since $\widetilde{E}\cF$ is not equivariantly contractible. Here $N_H^G$ denotes a norm \emph{functor} (see \cite{HHR_annals}) which is a left adjoint to the restriction $\res_H^G$ at the level of equivariant commutative ring spectra (algebras for the operad $\Comm$).

In the case of the circle, we will always take $\cF$ to be
the collection of all finite subgroups. 
Whereupon, $\widetilde{E}\cF$ has a strictly 
commutative model by 
\cite{JohnCouniversalComm}. The different behaviour in case of a circle comes from the fact that all proper subgroups are of infinite index in $SO(2)$ and the norm maps can only link subgroups $K$ and $H$ if $K\leq H$ is of finite index in $H$. We will use this observation in Section \ref{sec:algModTsp}, where we summarise the zig-zag of Quillen equivalences to the algebraic model for rational $\T$--spectra.

\subsection*{Contents of Part 1 of this paper}
Let $G=\T=SO(2)$ and write $\Eoneinfty$ for the \emph{non-equivariant} $\Einfty$--operad equipped with the trivial $\T$--action. 

The main theorem of Part 1 appears later in the paper as Theorem \ref{thm:fullcompare}.

\begin{thrA}
There is a zig-zag of Quillen equivalences
\[
\algin{\Eoneinfty}{  (\TSP_\bQ)   }
\ \simeq \
\algin{\Comm} d\mcA(\torus),
\]
where $ d\mcA(\torus)$ denotes the algebraic model for rational $\torus$-spectra from \cite{gre99, BGKS}.
\end{thrA}

In view of \cite{BlumbergHillNorms} $\Eoneinfty$--operad has the least commutative structure,
in particular it encodes the algebra structure without multiplicative norms.

Note that the operad $\Comm$ in the algebraic model doesn't encode any
additional structure beyond that of commutative algebras in the usual
sense,  and thus the category of $\Comm$--algebras in the algebraic model can only model the
rational na\"ive--commutative ring $\T$--spectra.

The present work is the first step towards obtaining an algebraic
model for $\T$--spectra that are equivariantly commutative in the
strongest sense. Notice, that in the case of a finite group $G$, there was a clear reason why it should be expected that the model categorical methods used to obtain the monoidal 
algebraic model only capture the lowest level of commutativity in the equivariant setting.
The methods for
rational $G$--spectra for a finite $G$ rely on the complete idempotent splitting using results of
Barnes \cite{barnessplitting}. The splitting operation does not preserve norms, and hence won't preserve more structured commutative objects (see for example McClure \cite{McClureTate}).

The case of $\T$ does not use the splitting techniques, instead it
uses an isotropy separation technique (known as the Tate square
or Hasse square). For suitable families of subgroups, this technique
preserves the strongest form of equivariant commutativity, so there is a good chance of obtaining an algebraic model for more structured commutative $\T$--ring spectra. In view of the present paper, that requires recognising additional structure on these commutative algebras in the algebraic model which are modelling strictly-commutative rational ring $\T$--spectra. This is work in progress.

\subsection*{Summary of the zig-zag of Quillen equivalences}

To illustrate the zig-zag of Quillen equivalences from \cite[Theorem 5.2.1]{BGKS} we present a diagram
of key steps. At the top we have our preferred model for rational $\mathbb{T}$-spectra (namely
the left Bousfield localization $L_{\bS_\bQ}\TSP$ of the category of
orthogonal  $\mathbb{T}$-spectra at the rational sphere spectrum). At the bottom we have the algebraic model
$d\mcA(\T)_{dual}$ with the dualizable model structure from \cite{barnesmonoidalmodelso2}.

The reader may wish to refer to this diagram now, but the notation will be introduced as we proceed.
In the diagram, left Quillen functors are placed on the left
and $\T=SO(2)$. A lift to the level of algebras over an $\Eoneinfty$--operad is given on
the left,  on the right there is an indication of the ambient
category.

\[\begin{array}{lr}
\xymatrix@R=1pc @C=7pc{
\algin{\Eoneinfty}{\TSP_\bQ}
\ar@<-1ex>[dd]_{\bbSfork \smashprod - } 
\\ \  \\
\algin{\Eoneinfty}{K_{\Top} \cell \bbSfork_{\Top} \leftmod}
\ar@<-1ex>[uu]_{\mathrm{pb}}
\ar@<+1ex>[dd]^{(-)^{\torus}}  
\\ \  \\
\algin{\Eoneinfty}{K_{\Top}^{\torus} \cell (\bbSfork_{\Top})^\torus \leftmod}
\ar@<+1ex>[uu]^{\zeta_{\#}}
\ar@<-1ex>[dd]_{\mathrm{change}}
\\ \ \\
\algin{\Comm}{K_{\Top}^{\torus} \cell (\bbSfork_{\Top})^\torus \leftmod}
\ar@<-1ex>[uu]_{\mathrm{of\ operads}}
\ar@<+0ex>[dd]^{\mathrm{of\ Quillen\ equivalences}} 
\\ \  \\
 \algin{\Comm}{K_t \cell \bbSfork_t \leftdmod}
\ar@<+0ex>[uu]^{\mathrm{zig-zag}}
\ar@<+0ex>[dd]^{\mathrm{of\ Quillen\ equivalences}} 
\\ \  \\
\algin{\Comm}{K_a  \cell \bbSfork_a \leftdmod} 
\ar@<+0ex>[uu]^{\mathrm{zig-zag}} 
\ar@<+1ex>[dd]^{\Gamma}  \\ \ \\
\algin{\Comm}{d \mcA(\mathbb{T})_{dual}}
\ar@<+1ex>[uu]^{\mu} 
}
& 
\xymatrix@R=1.03pc @C=7pc{
\TSP_\bQ
\ar@<-1ex>[dd]_{\bbSfork \smashprod - } & **[l] \mathrm{in\ } \TSP
\\ \  \\
 K_{\Top} \cell \bbSfork_{\Top} \leftmod
\ar@<-1ex>[uu]_{\mathrm{pb}}
\ar@<+1ex>[dd]^{(-)^{\torus}} & **[l] \mathrm{in\ } \TSP
\\ \  \\
K_{\Top}^{\torus} \cell (\bbSfork_{\Top})^\torus \leftmod
\ar@<+1ex>[uu]^{\zeta_{\#}}
\ar@<+0ex>[dd] & **[l] \mathrm{in\ Sp}
\\ \  \\ 
K_{\Top}^{\torus} \cell (\bbSfork_{\Top})^\torus \leftmod
\ar@<+0ex>[uu]^{\mathrm{Id}}
\ar@<+0ex>[dd]^{\mathrm{of\ Quillen\ equivalences}} & **[l] \mathrm{in\ Sp}
\\ \ \\
K_t \cell \bbSfork_t \leftdmod
\ar@<+0ex>[uu]^{\mathrm{zig-zag}}
\ar@<+0ex>[dd]^{\mathrm{of\ Quillen\ equivalences}} & **[l] \mathrm{in\ Ch}(\bQ)
\\ \  \\
K_a  \cell \bbSfork_a \leftdmod
\ar@<+0ex>[uu]^{\mathrm{zig-zag}}
\ar@<+1ex>[dd]^{\Gamma} & **[l] \mathrm{in\ Ch}(\bQ)
\\ \ \\
d \mcA(\mathbb{T})_{dual}
\ar@<+1ex>[uu]^{\mu} & **[l] \mathrm{in\ Ch}(\bQ)
}
\end{array}
\]

In the above the subscript `$\Top$' indicates that the
corresponding object has a topological origin, whereas the subscript
`$t$' indicates that the object is algebraic, but has been produced by
applying the results of \cite{shiHZ} and thus usually it does not have an explicit description. The subscript `$a$' indicates
that the object is algebraic in nature and has a small and explicit
description.  The symbols $\bbSfork_{(-)} $ refer to particular
$\pscr$-diagrams of rings, and the various categories $\bbSfork_{(-)} \leftmod$ are
diagrams of modules over a diagram of rings, see
Section \ref{sec:diagramsmodcat}. We use $\leftdmod$ to denote differential objects in graded modules, once we are in algebra.  These form model categories and  are cellularized (i.e., right Bousfield localized) at the sets of objects $K_{(-)}$, which at every level of
the diagram are the derived images of the usual stable generators
$\mathbb{T}/H_+$ of $\mathbb{T}$-spectra, where $H$ varies through closed subgroups of $\mathbb{T}$.

\subsection*{Contents of Part 2 of this paper: Elliptic cohomology}
Given an  elliptic curve $C$ over a $\Q$-algebra $K$, together with
some coordinate data, \cite{gre05} shows that one may construct an
$SO(2)$-equivariant elliptic cohomology associated to $C$. Indeed, 
one may write down an object $EC_a$ of the algebraic model $d\cA (\T)$
(the construction of $EC_a$ is recalled in 
  Subsection \ref{subsec:ECa}). 
 In view of the results of 
  \cite{gre99} and \cite{BGKS},  there is an associated rational   $\T$-spectrum 
  $EC_{\Top}$.  Moreover,  $d\cA (\T)$ is a symmetric 
monoidal category and $EC_a$ is visibly a commutative monoid.

It was asserted in \cite[Theorem 11.1]{gre05} that \cite{tnqcore} would 
(a) establish a monoidal Quillen equivalence and (b) that would  allow 
one to deduce that $EC_{\Top}$ is a commutative ring spectrum.
 It was then further stated (c) that   its  category of modules would be 
monoidally equivalent to a  derived category of quasi-coherent sheaves over the elliptic curve 
 $C$. 

 In the  event, \cite{tnqcore} did not prove a monoidal equivalence,
 but  we may refer instead to \cite{BGKS}. This shows we may take
$EC_{\Top}$ to be a ring spectrum. But even then,  the deduction of (b) 
from (a) requires substantial additional work, and Part 1 of this
paper establishes a version of it.  Finally, it is the purpose of 
Part 2 of the present paper to give a complete proof of (c). 

As noted above, $EC_a $ is visibly a commutative monoid, and  
therefore an algebra over $\Comm$, and by 
Theorem \ref{thm:fullcompare}, it  corresponds to an
$\Eoneinfty$-algebra $EC_{\Top}$ in rational  $\T$-spectra. We can
therefore consider the symmetric  monoidal category of modules.
 This is effectively what the author of \cite{gre05} had in mind when stating \cite[Theorem 11.1]{gre05}, and this structure suffices to give a natural symmetric 
monoidal product on $EC_{\Top}$--modules. 

To be precise on the geometric side,  we form a 2-periodic version
$P\cOC$ of the structure sheaf $\cOC$ to correspond to the fact that
$EC$ is 2-periodic. 
The non-empty open sets of the usual Zariski topology on $C$ consist of the complements  of arbitrary finite sets of points, but we use the torsion point (tp) topology in which
only points of finite order may be deleted.

The main theorem of Part 2 of the paper appears later as Theorem \ref{cor:main}.
\begin{thrm}
There is a symmetric monoidal Quillen equivalence 
\[
EC_{\Top}\leftmodin\TSP_\bQ \simeq 
\textnormal{quasi-coherent} \endash P\cOC \leftdmodin 
\textnormal{sheaves}/C_{tp}
\]
giving an equivalence of tensor triangulated categories 
\[
\ho( EC_{\Top} \leftmodin \TSP_\bQ) \simeq D(PC_{tp}). 
\]
\end{thrm}

We note that our result showing that $EC_{\Top}$ is an $\Eoneinfty$-algebra  falls short
of
showing $EC_{\Top}$ may be represented by a  commutative orthogonal
$\bT$-spectrum. It is now well known that commutative ring spectra  in the
category of orthogonal $\T$-spectra have additional structure 
(in particular, they admit multiplicative norms along finite index
inclusions) but this was not known to the author of \cite{gre05} in
2005.  We do not know whether $EC_{\Top}$ can be taken to be an 
$\EGinfty$-ring spectrum and hence  admit a  model as a commutative 
orthogonal $\T$-spectrum.

\subsection*{Notation}
From now on we will write $\torus$ for the group $SO(2)$. We also
stick to the convention of drawing the left adjoint above the right
one (or to the left, if drawn vertically) in any adjoint pair. We use $\ch(\bQ)$ for the category of chain
complexes of rational vector spaces, $\Sp$ for the category of
orthogonal spectra, $\GSp$ for the category of orthogonal $G$--spectra, and $\Sp^\Sigma$ for the category of symmetric spectra. We add a subscript $\bQ$ to indicate \emph{rational} ($G$-) spectra. From now on we will use notation $\Eoneinfty$ for non-equivariant operad $\Einfty$ given a trivial $\bT$-action and considered an an operad in $\bT$-spaces.

\subsection*{Acknowledgements} The authors would like to thank the Mathematisches Forschungsinstitut Oberwolfach and the Centre International de Rencontres Math\'ematiques in Luminy for providing  ideal environments to work on this project as part of Research in Pairs Programmes. The second author is grateful to EPSRC for support from EP/P031080/1. The third author was supported by a NWO Veni grant 639.031.757 during final stages of this work.

The authors would also like to thank the Isaac Newton Institute for Mathematical Sciences, Cambridge, for support and hospitality during the programme 
\emph{Homotopy Harnessing Higher Structures}
where work on this paper was completed. This work was supported by EPSRC grant number EP/K032208/1.

\part{Quillen equivalence with the algebraic model}

\section{Recollections}

We start this section by recalling several results on lifting various model structures to the categories of algebras over certain operads in equivariant stable setting.

\subsection{Model structures on algebras over operads}

\begin{definition} Let $\mcO$ be an operad in $G$-topological spaces, which has a trivial $G$-action and its underlying non-equivariant operad is an $\einf$-operad in topological spaces. We will call such an operad $\Eoneinfty$.
\end{definition}

\begin{example}\label{ex:operads}
Let $U$ denote a $G$-universe, i.e. a countably infinite-dimensional real $G$-inner product space which contains each finite dimensional sub-representation infinitely often. The \emph{linear isometries operad} $\cL(U)$ is a $G$-operad such that $\cL(U)(n):= \cL(U^n,U)$, where $ \cL(U^n,U)$ denotes non-equivariant linear isometries from $U^n$ to $U$. It is a $G\times \Sn$ space by conjugation and diagonal action. The identity map $U\lra U$ is the distinguished element of $\cL(U)(1)$ and the structure maps are given by composition.

If $U$ is a \textbf{$G$-fixed} $G$-universe then $\cL(U)$ is an example of an  $\Eoneinfty$-operad.

\end{example}

Since the category of $G$-spectra is tensored over $G$-spaces we can consider $\mcO$-algebras in $G$-spectra, where $\mcO$ is a $G$-operad (an operad in $G$-spaces). In particular this applies when $\mcO$ is an $\Eoneinfty$-operad in $G$-spaces.

In the next proposition, which follows from  \cite[Proposition A.1]{BlumbergHillNorms}, we consider the category of orthogonal $G$-spectra with the positive stable model structure of \cite[Section III.5]{mm02}.

\begin{proposition}\label{prop:liftingToOalg} For any $\Eoneinfty$-operad $\mcO$ such that each $\mcO(n)$ has a homotopy type of a $(G\times \Sn)$--CW complex there exists a right-lifted model structure on $\mcO$-algebras in  orthogonal $G$--spectra (i.e. the weak equivalences and fibrations are created in the category of orthogonal $G$--spectra with the positive stable model structure).
\end{proposition}

\begin{lemma}\label{lem:adjunctionlift}
Let $(L,R)$ be a $\bV$-enriched strong symmetric monoidal adjunction between $\bV$-tensored
symmetric monoidal categories $\bC$ and $\bD$.

Let $\mcO$ be an operad in $\bV$,
then $L$ and $R$ extend to functors of $\mcO$-operads in $\bC$ and $\bD$, and we obtain
a square of adjunctions
\[
\xymatrix{
\algin{\mcO}{\bC} \ar@<0.5ex>[d]^{U_\bC}  \ar@<0.5ex>[r]^L
&
\algin{\mcO}{\bD} \ar@<0.5ex>[d]^{U_\bD}   \ar@<0.5ex>[l]^R
\\
\bC \ar@<0.5ex>[r]^L \ar@<0.5ex>[u]^{F_{\mcO}}
&
\bD  \ar@<0.5ex>[l]^R \ar@<0.5ex>[u]^{F_{\mcO}}
}
\]
which is commutative in the sense that the square of left adjoints
commutes and the square of right adjoints commutes. 
\end{lemma}

Now we consider the case when the original $L,R$-adjunction is a
Quillen equivalence and give criteria for it to lift to a
Quillen equivalence of $\mcO$-algebras. 

\begin{lemma}\label{lem:liftQE}\cite[Lemma 3.6]{BGKfinite} Suppose that $L$ is a strong symmetric monoidal functor tensored over $\bV$, $L$ and $R$ form a Quillen equivalence (at the level of categories $\bC$ and $\bD$) and
that the categories of $\mcO$--algebras in $\bC$ and $\bD$ have
right--lifted cofibrantly generated model structures from the ones on
$\bC$ and $\bD$ respectively. 

If  $U_\bC$ preserves cofibrant objects then the lifted adjoint pair
$L$, $R$ at the level of $\mcO$--algebras is a Quillen equivalence:
\[ \algin{\mcO}{\bC}\simeq  \algin{\mcO}{\bD} .\] 
\end{lemma}

We will apply this result numerous times in the rest of the paper in case where $\mcO$ is an $\Eoneinfty$--operad and categories $\bC$ and $\bD$ are built from equivariant orthogonal spectra, as we explain in the next section.


\subsection{Diagrams of model categories}\label{sec:diagramsmodcat}
Since the methods used to obtain an algebraic model for rational $\torus$-spectra are substantially different than the ones for a finite group we recall here the basic building blocks from \cite{BGKS} and \cite{tnqcore}, i.e. diagrams of model categories. This idea has been studied in some detail in \cite{huettemann-roendigs} and in \cite{gsmodules}, among others. In this section we introduce the
relevant structures and leave most of the proofs to the references.
We will only use one shape of a diagram, the pullback diagram $\pscr$:
\[
\bullet \longrightarrow \bullet \longleftarrow \bullet, 
\]
where the direction of the arrows refers to the left adjoints.
Pullbacks of model categories are also considered in detail in
\cite{berg11}. 

\begin{definition}
A \textbf{$\pscr$--diagram of model categories} $R^\bullet$
is a pair of Quillen pairs
\[
\begin{array}{rcl}
L:
\mcA
&
\adjunct
&
\mcB : R \\
F:
\mcC
&
\adjunct
&
\mcB : G \\
\end{array}
\]
with $L$ and $F$ the left adjoints.
We will usually draw this as the diagram below.
\[
\xymatrix{
\mcA
\ar@<+1ex>[r]^L
&
\mcB
\ar@<+0.5ex>[l]^R
\ar@<-0.5ex>[r]_G
&
\mcC
\ar@<-1ex>[l]_F
}
\]
\end{definition}

A standard example, which is of the main interest to this paper, comes from a $\pscr$--diagram of rings
$\Rfork=(R_1 \overset{f}{\rightarrow} R_2 \overset{g}{\leftarrow} R_3)$. Using the adjoint pairs
of extension and restriction of scalars we obtain a $\pscr$--diagram
of model categories $R^\bullet$ as below.
\[
\xymatrix@C+0.3cm{
R_1 \leftmod
\ar@<+1ex>[r]^{R_2 \otimes_{R_1} -}
&
R_2 \leftmod
\ar@<+0.5ex>[l]^{f^*}
\ar@<-0.5ex>[r]_{g^*}
&
R_3 \leftmod.
\ar@<-1ex>[l]_{R_2 \otimes_{R_3} -}
}
\]

\begin{definition}\label{def:generalisedDiag}
Given a $\pscr$--diagram of model categories $R^\bullet $
we can define a new category, $R^\bullet \leftmod$.
The objects of this category are pairs of morphisms,
$\alpha \co La \to b$ and $\gamma \co Fc \to b$ in $\mcB$.
We usually abbreviate a pair $(\alpha \co La \to b, \gamma \co Fc \to b)$
to a quintuple $(a,\alpha, b, \gamma, c)$.

A morphism in $R^\bullet \leftmod$ from
$(a,\alpha, b, \gamma, c)$ to $(a',\alpha', b', \gamma', c')$
is a triple of maps
$x \co a \to a'$ in $\mcA$,
$y \co b \to b'$ in $\mcB$,
$z \co c \to c'$ in $\mcC$ such that we have a commuting diagram in $\mcB$
\[
\xymatrix{
La \ar[r]^\alpha
\ar[d]^{Lx}
& b
\ar[d]^y
& Fc \ar[l]_\gamma
\ar[d]^{Fz} \\
La' \ar[r]^{\alpha'}
& b'
& Fc' \ar[l]_{\gamma'}
}
\]
\end{definition}

Note that we could also have defined an object as a sequence
$(a,\bar{\alpha}, b, \bar{\gamma}, c)$.
where $\bar{\alpha} \co a \to Rb$ is a map in $\mcA$
and $\bar{\gamma} \co c \to Gb$ is a map in $\mcC$.

\begin{notation}
We use $\Rfork\leftmod$ to denote a category of modules over a diagram of rings $\Rfork$. This is a special case of the category $R^\bullet\leftmod$ from Definition \ref{def:generalisedDiag}.
\end{notation}

We say that a map $(x,y,z)$ in $R^\bullet \leftmod$ is an objectwise
cofibration if $x$ is a cofibration of $\mcA$,
$y$ is a cofibration of $\mcB$ and
$z$ is a cofibration of $\mcC$.
We define objectwise weak equivalences similarly.

\begin{lemma}\cite[Proposition 3.3]{gsmodules}
Consider a $\pscr$--diagram of model categories $R^\bullet $ as below,
with each category cellular and proper.
\[
\xymatrix{
\mcA
\ar@<+1ex>[r]^L
&
\mcB
\ar@<+0.5ex>[l]^R
\ar@<-0.5ex>[r]_G
&
\mcC
\ar@<-1ex>[l]_F
}
\]
The category $R^\bullet \leftmod$ admits a cellular
proper model structure with cofibrations and weak
equivalences defined objectwise. This is called the {\em diagram injective} model structure.
\end{lemma}

Whilst there is also a \emph{diagram projective} model structure,
in this paper we only use the diagram injective model structure (and cellularizations thereof)
on diagrams of model categories.

Now consider maps of $\pscr$--diagrams of model categories.
Let $R^\bullet $ and $S^\bullet$ be two  $\pscr$-diagrams,
where $R^\bullet $ is as above and $S^\bullet$ is given below.
\[
\xymatrix{
\mcA'
\ar@<+1ex>[r]^{L'}
&
\mcB'
\ar@<+0.5ex>[l]^{R'}
\ar@<-0.5ex>[r]_{G'}
&
\mcC'
\ar@<-1ex>[l]_{F'}
}
\]
Now we assume that we have Quillen adjunctions as below
such that $P_2 L$ is naturally isomorphic to
$L' P_1$ and $P_2 F$ is naturally isomorphic to
$F' P_3$.
\[
\begin{array}{rcl}
P_1:
\mcA
&
\adjunct
&
\mcA' : Q_1 \\
P_2:
\mcB
&
\adjunct
&
\mcB' : Q_2 \\
P_3:
\mcC
&
\adjunct
&
\mcC' : Q_3 \\
\end{array}
\]
We then obtain a Quillen adjunction $(P,Q)$ between
$R^\bullet \leftmod$ and $S^\bullet \leftmod$.
For example, the left adjoint $P$ takes the object
$(a,\alpha,b,\gamma,c)$ to
$(P_1 a, P_2 \alpha, P_2 b, P_2 \gamma, P_3 c)$.
The commutativity assumptions ensure that this is an
object of $S^\bullet \leftmod$. It is easy to see the following

\begin{lemma}\label{lemma:QEondiag}
If the Quillen adjunctions $(P_i, Q_i)$ are
Quillen equivalences then the adjunction $(P,Q)$ between
$R^\bullet \leftmod$ and $S^\bullet \leftmod$ is a
Quillen equivalence.
\end{lemma}

Now we turn to monoidal considerations.
There is an obvious monoidal product for $R^\bullet \leftmod$,
provided that each of $\mcA$, $\mcB$ and $\mcC$ is monoidal
and that the left adjoints $L$ and $F$ are strong monoidal.
\[
(a,\alpha, b, \gamma, c) \smashprod (a',\alpha', b', \gamma', c')
:=(a \smashprod a', \alpha \smashprod \alpha',
b \smashprod b',\gamma \smashprod \gamma',c \smashprod c')
\]
Let $S_\mcA$ be the unit of $\mcA$,
$S_\mcB$ be the unit of $\mcB$ and let
$S_\mcC$ be the unit of $\mcC$.
Since $L$ and $F$ are monoidal, we have maps
$\eta_\mcA \co L S_\mcA \to S_\mcB$
and $\eta_\mcC \co F S_\mcC \to S_\mcB$.
The unit of the monoidal product on
$R^\bullet \leftmod$ is
$(S_\mcA, \eta_\mcA, S_\mcB, \eta_\mcC, S_\mcC)$.

It is worth noting that this category has an internal function object
when $\mcA$, $\mcB$ and $\mcC$ are closed monoidal categories and thus itself is closed.

\begin{lemma}
Consider a $\pscr$--diagram of model categories $R^\bullet $ such that each
vertex is a cellular monoidal model category.
Assume further that the two adjunctions of the diagram are strong monoidal
Quillen pairs.
Then $R^\bullet \leftmod$ is also a monoidal model category.
If each vertex also satisfies the monoid axiom,
so does $R^\bullet \leftmod$.
\end{lemma}
\begin{proof}
Since the cofibrations and weak equivalences are defined objectwise,
the pushout product and monoid axioms hold provided they do so
in each model category in the diagram $R^\bullet $.
\end{proof}

We can also extend our monoidal considerations to maps
of diagrams.
Return to the setting of a map $(P,Q)$
of $\pscr$--diagrams from $R^\bullet $ to $S^\bullet$ as described above.
If we assume that each of the adjunctions
$(P_1, Q_1)$, $(P_2, Q_2)$ and $(P_3, Q_3)$
is a symmetric monoidal Quillen equivalence, then we see that
$(P,Q)$ is a symmetric monoidal Quillen equivalence.

\begin{lemma}Consider a $\pscr$--diagram of model categories $R^\bullet $ such that each
vertex is tensored over a symmetric monoidal category $\bV$ and  the two left adjoints of the diagram are also tensored over $\bV$. Then $R^\bullet \leftmod$ is tensored over $\bV$, using the formula
$$v\otimes (a,\alpha, b, \gamma, c):= (v\otimes a, v\otimes \alpha,v\otimes b, v\otimes \beta, v\otimes c).$$
\end{lemma}

\begin{lemma}Under the assumptions from the lemma above, if $\mcO$ is an operad in $\bV$ and both left adjoints of the diagram $R^\bullet $ are strong symmetric monoidal, then the category of $\mcO$--algebras in $R^\bullet \leftmod$ is equivalent to the category $(\algin{\mcO}{R^\bullet}) \leftmod$, where  $\algin{\mcO}{R^\bullet}$ is the following diagram
\[
\xymatrix{
\algin{\mcO}{\mcA}
\ar@<+1ex>[r]^L
&
\algin{\mcO}{\mcB}
\ar@<+0.5ex>[l]^R
\ar@<-0.5ex>[r]_G
&
\algin{\mcO}{\mcC}
\ar@<-1ex>[l]_F
}
\]
\end{lemma}


\subsection{Cellularization}
A cellularization of a model category
is a right Bousfield localization at a set of objects.
Such a localization exists by \cite[Theorem 5.1.1]{hir03}
whenever the model category is right proper and cellular.
When we are in a stable context the results of \cite{brstable}
can be used.

In this subsection we recall the notion of cellularization when $\mcC$ is stable.

\begin{definition}
Let $\mcC$ be a stable model category and $K$ a stable set of objects of $\mcC$, i.e. such that a class of $K$--cellular objects of $\mcC$ is closed under desuspensions.\footnote{Note that this class is always closed under suspensions.}
We say that a map $f \co A \longrightarrow B$ of $\mcC$ is a \textbf{$K$--cellular equivalence} if
the induced map
\[
[k,f]^\mcC_*: [k,A]^\mcC_* \longrightarrow [k,B]^\mcC_*
\]
is an isomorphism of graded abelian groups for each $k \in K$. An object $Z \in \mcC$ is said to be
\textbf{$K$--cellular} if
\[
[Z,f]^\mcC_*: [Z,A]^\mcC_* \longrightarrow [Z,B]^\mcC_*
\]
is an isomorphism of graded abelian groups for any $K$--cellular equivalence $f$.
\end{definition}

\begin{definition}
A \textbf{right Bousfield localization} or \textbf{cellularization} of $\mcC$ with respect to
a set of objects $K$ is a model structure $K \cell \mcC$ on $\mcC$ such that
\begin{itemize}[noitemsep]
\item the weak equivalences are $K$-cellular equivalences
\item the fibrations of $K \cell \mcC$ are the fibrations of $\mcC$
\item the cofibrations of $K \cell \mcC$ are defined via left lifting property.
\end{itemize}
\end{definition}

By \cite[Theorem 5.1.1]{hir03}, if $\mcC$ is a right proper, cellular model category
and $K$ is a set of objects in $\mcC$, then the cellularization of $\mcC$ with respect to $K$, $K \cell \mcC$,
exists and is a right proper model category.
The cofibrant objects of $K \cell \mcC$
are called \textbf{$K$--cofibrant} and are precisely the
$K$--cellular and cofibrant objects of $\mcC$.
Moreover, if $\mcC$ is a stable model category and $K$ is a stable set of cofibrant objects (see \cite[Definition 4.3]{brstable}), then $K \cell \mcC$ is also cofibrantly generated, and the set of generating acyclic cofibrations is the same as for $\mcC$ (see \cite[Theorem 4.9]{brstable}).

\section{Preliminaries}

We now proceed to lifting model structures to the categories of algebras over $\Eoneinfty$-operads in this new setting of diagrams of model categories.

\subsection{Lifting model structures}

If $\bC$ is a model category we want to right--lift the model structure from $\bC$ to
the category $\algin{\mcO}{\bC}$ of $\mcO$-algebras in $\bC$ using 
the right adjoint $U$, i.e. we want the weak equivalences and fibrations to be created by $U$.
In our case, $\bC$ will be some category of $R$-modules in $\TSP$ (or spectra) 
with weak equivalences those maps which forget to rational equivalences of $\TSP$.

First we recall Kan's result for right lifting model structures in this setting.

\begin{lemma}\cite[Theorem 11.3.2]{hir03}
\label{lem:liftedmodel} Suppose $\bC$ is a cofibrantly generated model category with a set of generating acyclic cofibrations $J$.  If
\begin{itemize}
\item the free $\mcO$--algebra functor $F_{\mcO}$ preserves small objects (or the forgetful functor $U$ preserves filtered
  colimits) and
  \item every transfinite composition of pushouts (cobase extensions)
  of elements of $F_{\mcO}(J)$ is sent to a weak equivalence
  of $\bC$ by $U$,
\end{itemize}
then the model structure on $\bC$ may be lifted using the right
adjoint to give a cofibrantly generated model structure on
$\algin{\mcO}{\bC}$. The functor $U$ then creates fibrations and weak
equivalences.
\end{lemma}

The following remark tells us that cellularisations interact well with lifted model structures. 

\begin{remark}\label{rem:cellularisation_lift}
Let $\bC$ be a cofibrantly generated model structure and $\mcO$ be an operad such that 
the free $\mcO$--algebra functor $F_{\mcO}$ preserves small objects and such that the model structure on $\bC$ 
right lifts to $\mcO$--algebras in $\bC$.
Assume that for a set of objects $K$ the cellularisation $K \cell \bC$ of the model structure on $\bC$ exists. 
Then there exists a right--lifted model structure on $\mcO$--algebras in $\bC$ from $K \cell \bC$.

This follows directly from Kan's lifting lemma and the fact that the 
generating acyclic cofibrations for $K \cell \bC$ are the same as the ones for $\bC$.
\end{remark}

We will lift various model structures on $R$--modules in $\TSP_\bQ$ to
the level of algebras over an $\Eoneinfty$--operad $\mcO$. 
We first consider the categorical setting. 
Since $R$ is a $\Comm$--algebra we get a distributivity law $$F_{\mcO}(R\wedge -) \lra R\wedge F_{\mcO}(-),$$
and hence $$\algin{\mcO}{R \leftmodin \TSP}\cong R \leftmodin \algin{\mcO} \TSP.$$
Therefore the functor $R \sm -$ lifts to a functor at the level of $\mcO$-algebras as in the diagram below. 
In the following, $U$ denotes 
forgetful functors and $F_{\mcO,R}$ is defined by 
$X \mapsto \vee_{n \geqslant 0} \mcO(n)_+ \sm_{\Sn} X^{\sm_R^n}.$ 
\[
\xymatrix@C+1cm{
\algin{\mcO}{R \leftmodin \TSP}
\ar@<-1ex>[r]_-{U}
\ar@<+1ex>[d]^{U}
&
\ar@<-1ex>[l]_-{R \sm -}
\algin{\mcO}{\TSP}
\ar@<+1ex>[d]^{U}
\\
R \leftmodin \TSP  
\ar@<-1ex>[r]_-{U}
\ar@<+1ex>[u]^{F_{\mcO,R}}
&
\ar@<-1ex>[l]_-{R \sm -}
\TSP
\ar@<+1ex>[u]^{F_{\mcO}}
}
\]

Now considering model categories, we want to lift the positive stable model structure on rational $\T$--spectra to the other three corners. 
Lifting to $R$-modules is standard and lifting to $\mcO$-algebras
is Proposition \ref{prop:liftingToOalg}. Thus all that remains is the 
top left corner.

Since the homotopy theory of $\mcO$-algebras does not depend on the choice of the $\Eoneinfty$-operad $\mcO$ which at level $n$ has a homotopy type of $(\bT\times \Sigma_n)$-CW complex, from now on we will use the generic notation $\Eoneinfty$-algebras.

\begin{lemma}\label{lem:tspectraRmodules}  For $R$ a commutative ring $\torus$-spectrum
there is a cofibrantly generated model structure on
$\Eoneinfty$-algebras in $R\leftmod$ where the weak equivalences
are those maps which forget to rational stable equivalences of $\torus$-spectra.
Furthermore, the forgetful functor to $R\leftmod$ preserves cofibrant objects.
\end{lemma}
\begin{proof} 

We adapt the proof of \cite[Theorem 4.4]{BGKfinite}. 
A corollary of that theorem states that there is a lifted model structure on  $\Eoneinfty$--algebras in $\TSP_\bQ$. 
We extend this to 
$\Eoneinfty$--algebras in $R \leftmodin \TSP_\bQ$. Specifically, we show that the 
adjunction $(F_{\Eoneinfty,R}, U)$ can be used to lift the model structure on 
$R \leftmodin \TSP_\bQ$ to $\Eoneinfty$--algebras in $R \leftmodin  \TSP_\bQ$. 
 
As in the reference, this result has two parts, 
the first part is about the interaction of pushouts,
sequential colimits and $h$-cofibrations (also known as the Cofibration Hypothesis).
This is well known and uses a standard technique of describing
pushouts of $\Eoneinfty$-algebras as sequential colimits in the underlying category 
(in this case $R \leftmod$).
This argument appears in several places, one of the earlier occurrences is Elmendorf et al \cite[Chapters VII and VIII]{EKMM97}  where they construct model structure of commutative algebras and localisations of commutative algebras. A more recent description is given by Harper and Hess in  \cite[Proposition 5.10]{harperhess13}.

The second part is that every map built using pushouts (cobase extensions) and
sequential colimits from $F_{\Eoneinfty,R} (R \sm J)$ 
(i.e. maps of the form $F_{\Eoneinfty,R} (R \sm j)$, for $j$ a generating acyclic
cofibration of $ \TSP_\bQ$) is a rational equivalence. It can be broken into a number of steps.
\begin{description}
\item[Step 1] If $f \co X \to Y$ is a
generating acyclic cofibration of the rational model structure on 
$R$-modules then
$F_{\Eoneinfty,R} f$ is a rational equivalence and a $h$-cofibration of underlying \mbox{$\T$-spectra}.

\item[Step 2] Any pushout (cobase extension) of
$F_{\Eoneinfty,R} f$ is a rational equivalence and a $h$-cofibration of underlying $\T$-spectra.

\item[Step 3] Any sequential colimit of such maps is a rational equivalence.
\end{description}

To prove Step 1 we first identify the generating acyclic cofibrations as 
maps of the form $R \sm g \co R \sm Z \to R \sm Z'$ where $g \co Z \to Z'$
is a acyclic generating cofibration of $\TSP_\bQ$ and hence has cofibrant domain
and codomain. We know that 
$F_{\Eoneinfty,R} (R \sm g) = R \sm F_{\Eoneinfty} g$. 
By  \cite[Theorem 4.4]{BGKfinite}, 
$F_{\Eoneinfty} g$ is a rational acyclic cofibration of cofibrant \mbox{$\torus$--spectra}. 
Smashing such a map with a spectrum gives a rational acyclic $h$-cofibration of $\torus$--spectra that in $R \leftmod$
is a cofibration of cofibrant objects. 
Step 2 and Step 3 follow as in the classical case. 

By \cite[Theorem 5.18]{harperhess13} the forgetful functor
\[
U \colon \algin{\Eoneinfty}{R \leftmodin \TSP_\bQ} 
\longrightarrow 
R \leftmodin \TSP_\bQ 
\]
preserves cofibrant objects. 
\end{proof}

\begin{lemma}\label{lem:generalised_diag_lift} Suppose there is a right-lifted model structure to $\algin{\mcO}{\bC}$ from each category ${\bC}$ used to built the diagram category ${R^\fork}\leftmod$ separately and $$U:\algin{\mcO}{R^\fork} \leftmod \lra {R^\fork} \leftmod$$ commutes with filtered colimits. Then there is a right-lifted model structure from the diagram injective model structure on ${R^\fork} \leftmod$ to $\algin{\mcO}{R^\fork} \leftmod$.
\end{lemma}
\begin{proof}
If each separate lift exists, then the conditions of Kan's result are satisfied for the (generalised) diagram category, since the free $\mcO$-algebra functor, acyclic cofibrations, weak equivalences, transfinite compositions and pushouts are defined objectwise.
\end{proof}

As well as $\Eoneinfty$-algebras on $\TSP$ and $\Sp$ we will need
$\Comm$-algebras on various model categories of non-equivariant spectra. 
Specifically we will need to consider 
$\Comm$-algebras in the rationalised positive stable model structure on 
orthogonal spectra, 
$\Comm$-algebras in the rationalised positive stable model structure on 
symmetric spectra, 
and 
$\Comm$-algebras in the positive stable model structure on 
$H \bQ$--modules in symmetric spectra.  
We give a proof that the various lifted model structures exist for
the orthogonal spectra case, the remaining cases are similar. 

\begin{lemma}\label{lem:R-mod_lift} 
Let $R$ be a commutative ring spectrum.
If we equip orthogonal spectra with 
the positive rational stable model structure 
then there are lifted model structures on the categories of 
$\Comm$-algebras in $R$-modules in $\Sp_\bQ$ and  
$\Eoneinfty$-algebras in $R$-modules in $\Sp_\bQ$.
\end{lemma}
\begin{proof}
The category of commutative algebras in $R$-modules
is the category of commutative $R$-algebras, 
which in turn is the category of commutative algebras
under $R$. Hence the
model structure for $\Comm$-algebras in $\Sp_\bQ$
\cite[Theorem A.2]{BGKfinite}
gives a model structure on 
$\Comm$-algebras in $R$-modules in $\Sp_\bQ$.

For the $\Eoneinfty$-case we use a non-equivariant version 
of Lemma \ref{lem:tspectraRmodules}. 
\end{proof}

\section{An algebraic model for rational \texorpdfstring{$\bT$}{T}--spectra}\label{sec:algModTsp}

We start by summarising the classification of 
rational $\bT$--spectra in terms of an
algebraic model. We will only outline the strategy here, for details and definitions we refer the reader to \cite{BGKS} and \cite{tnqcore}.

\begin{remark}\label{rmk:threechoices}
We choose to work with a simplified version of the proof presented in \cite{BGKS}. The simplification arises from very recent work of the 
second author, \cite{JohnCouniversalComm}
which proves that the ring $\bT$-spectrum 
$\widetilde{E}F$ has a genuinely commutative model.

This means that instead of working with localisations of model categories used in \cite{BGKS}, we work with diagrams of modules over a diagram of \emph{genuinely} commutative ring $\bT$--spectra. This approach is simpler,
and more closely follows the approach of
\cite{tnqcore}.

We remark here, that any one of the three zig-zags: the one presented in this paper, the one used in \cite{BGKS} or the one using $\cL(1)$-modules in $\bT$-orthogonal spectra suggested in \cite{tnqcore}, can be used to 
prove that there is a zig-zag of Quillen equivalences
\[
\algin{\Eoneinfty}{\TSP_\bQ}
\ \simeq \ 
\algin{\Comm}{d \mcA(\mathbb{T})_{dual}}
\]
\end{remark}

The Tate square for the family of finite subgroups is the homotopy
pullback square
$$\xymatrix{
\bbS \ar[r] \ar[d] & \siftyV{\bT}\ar[d] \\
D\efp \ar[r] & \siftyV{\bT}\sm D\efp 
}$$
of (genuinely) commutative ring $\T$-spectra (Note that $\siftyV{\bT}=\widetilde{E}F$ has a commutative model in rational orthogonal $\bT$-spectra by \cite{JohnCouniversalComm}). 
We then omit the copy of $\bbS$ at the top left, to leave the fork 
\[
\bbSfork_{\Top}=
\left(
\vcenter{\xymatrix{ & \siftyV{\bT}\ar[d] \\
D\efp \ar[r] & \siftyV{\bT}\sm D\efp 
}}\right) .
\]
The Cellularization Principle shows that, the category of $\T$-spectra can then be recovered from modules over 
$\bbSfork_{\Top}$ by taking a pullback construction \cite[Proposition 4.1]{gsmodules} (see also 
\cite[Proposition 4.1]{tnqcore}, 
\cite[Proposition 6.2]{gscell}): 
\[
\TSP_\bQ=  \bS \leftmodin \TSP_\bQ \simeq 
K_{\Top} \cell \bbSfork_{\Top} \leftmodin \TSP_\bQ.
\]
The cellularization here is at the derived image of the generators for $\TSP_\bQ$, which we denote by $K_{\Top}$. Taking $\T$-fixed points gives the diagram of rings
(arranged on a line for typographical convenience)
\[
(\bbSfork_{\Top})^\torus
=
\left(
D\efp^\bT
\longrightarrow
(\siftyV{\bT}\sm D\efp)^{\bT}
\longleftarrow (\siftyV{\bT})^{\bT}
\right).
\]
If $M$ is a module over the ring $\torus$--spectrum $R$, 
then $M^\torus$ is a module over the (non-equivariant) 
ring spectrum $R^\torus$. Taking $\torus$-fixes points
at every place of the diagram gives a Quillen equivalence
\[
\adjunction{\zeta_{\#}}{(\bbSfork_{\Top})^\torus \leftmod}{\bbSfork_{\Top} \leftmodin \TSP_\bQ}{(-)^\torus}
\]
by \cite[Section 3.3]{BGKS} or \cite[Section 7.A]{tnqcore}.
Let $K_{\Top}^\torus$ be the derived image of the cells 
$K_{\Top}$ under this adjunction. Then we have a 
symmetric monoidal Quillen equivalence 
\[
\adjunction{\zeta_{\#}}
{K_{\Top}^\torus \cell (\bbSfork_{\Top})^\torus \leftmod}
{K_{\Top} \cell \bbSfork_{\Top} \leftmodin \TSP_\bQ}{(-)^\torus}.
\]

Rational orthogonal spectra are symmetric 
monoidally Quillen equivalent to 
rational symmetric spectra in simplicial sets
$\Sp_\bQ^\Sigma$, which in turn are 
symmetric monoidally Quillen equivalent to 
$H \bQ$--modules in symmetric spectra.
Hence we can apply the 
results of \cite{shiHZ} to obtain a diagram $\bbSfork_{t}$ of 
commutative ring objects in rational chain complexes
and symmetric monoidal Quillen equivalences
\[
(\bbSfork_{\Top})^\torus \leftmod
\ \simeq \ 
\bbSfork_{t} \leftdmod 
\quad \quad 
K_{\Top}^\torus \cell (\bbSfork_{\Top})^\torus \leftmod
\ \simeq \ 
K_{t} \cell \bbSfork_{t} \leftdmod \\
\]
where $K_{t}$ is the derived image of the cells 
$K_{\Top}^\torus$, see
\cite[Section 3.4]{BGKS} or \cite[Section 8]{tnqcore}. 
Here, $\leftdmod$ denotes differential objects in graded modules. 
The results of \cite{shiHZ} give isomorphisms between 
the homology of these rational chain complexes 
in terms of homotopy groups of the original spectra. 

There is a product splitting 
\[
D\efp \simeq \prod_n D\elr{n}
\]
where the single isotropy space $\elr{n}$ is defined 
by the cofibre sequence
\[
E[H\subset C_n]_+\lra E[H\subseteq  C_n]_+\lra \elr{n}.
\]
We may then calculate 
\[
\pi_*(D\elr{n}^\torus)
=
\pi_*^\torus (D\elr{n}) \cong 
H^*(BS^1/C_n) = \bQ[c],
\quad \quad 
\pi_*((\siftyV{\bT})^{\bT}) 
= \bQ
\]
As these rings are polynomial, we use formality arguments 
to see that 
$\bbSfork_{t}$ is quasi-isomorphic (as diagrams of commutative rings)
to the diagram 
\[
\bbSfork_a=\left( 
\vcenter{\xymatrix{ & \Q\ar[d] \\
\cOcF  \ar[r] & \cEi \cOcF
}}\right) 
\]
of graded rings. Recall that 
$\cOcF \cong \prod_n \Q[c] $
and 
\[
\cEi \cOcF = \colim_{V^{\T}=0}\Sigma^V \cOcF
\]
where $\Sigma^V\cOcF=\prod_n\Sigma^{|V^{C_n}|}\Q [c]$. 
For typographical convenience we arrange this along a line
$$\cOcF \stackrel{l}\lra \cEi \cOcF \stackrel{\iota}\lla \Q, $$
and modules over this diagram are of the form 
$$N \stackrel{k}\lra P \stackrel{m}\lla V$$ 
where $k$ is the map of $\cOcF$-modules and $m$ is a map of $\Q$-modules.

Formality arguments allow us to simplify the diagram of commutative rings $\bbSfork_t$ to a quasi-isomorphic diagram of commutative rings $\bbSfork_a$. 
Another sequence of formality arguments gives a 
simpler set of cells $K_a$ defined in 
\cite[Lemma 4.2.2]{BGKS}.
Hence we have zig-zags of symmetric monoidal Quillen equivalences 
\[
\bbSfork_a \leftdmod \ \simeq \ \bbSfork_t \leftdmod
\quad \quad 
K_{a} \cell \bbSfork_a \leftdmod \ \simeq \ 
K_{t} \cell \bbSfork_t \leftdmod
\]
see \cite[Section 9]{tnqcore}.

Let $\cA (\T)$ be the subcategory of $\Rafork$-modules
where the structure maps induce isomorphisms $l_*N\cong P$ and $\iota_*V\cong 
P$. We use notation $d\cA (\T)$ for the category of objects of $\cA (\T)$ equipped with a differential. This categories are described in detail in \cite{BGKS}.

We give the differential module categories of the individual rings in $\Rafork$ the
(algebraic) projective model structures and $\Rafork \leftdmod$ the diagram
injective model structure. We give the category $d\cA(\T)$ of differential  objects of
$\cA (\T)$ the dualizable model structure of Barnes
\cite{barnesmonoidalmodelso2}: the essential fact is that this is a monoidal model
structure.  

The final stage in our sequence of Quillen equivalences is
to remove the cellularisation.
There is a symmetric monoidal Quillen equivalence 
\[
K_a \cell \Rafork \leftdmod \simeq d\cA (\T),
\]
by \cite[11.5]{tnqcore}.

\section{An algebraic model for rational na\"ive-commutative ring \texorpdfstring{$\torus$}{T}--spectra}

In this section we construct our sequence of Quillen equivalences between $\Eoneinfty$-algebras in
rational $\torus$-spectra and commutative algebras in the algebraic model. The strategy is to replace
$\Eoneinfty$ by $\Comm$ as soon as we have moved from equivariant spectra to  (modules over $(\bS^\fork_{\Top})^{\bT}$ in) non-equivariant spectra. We then apply the results of
Richter and Shipley \cite{richtershipley} to compare commutative algebras in rational spectra to
commutative algebras in rational chain complexes.

Before we check the compatibility of the zig-zag of Quillen equivalences from Section \ref{sec:algModTsp} with algebras over $\Eoneinfty$-operads, we note that this zig-zag of Quillen equivalences works for positive stable model structures.

\begin{remark} The model structures and Quillen equivalences of Section \ref{sec:algModTsp}
relating rational $\torus$-spectra to the algebraic model $d\mcA(\torus)$ can all be obtained using the
positive stable model structure on the various categories of spectra.
Moreover, the adjunction
\[
\adjunction{H\bQ \wedge -}{\Sp_\bQ^\Sigma}{H\bQ \leftmod}{U}
\]
is a Quillen equivalence when the left side is considered with the positive stable model structure and the right hand side is considered with the positive flat stable model structure of \cite{Shipley_convenient}. This model structure on $H\bQ \leftmod$ is the starting point for the work in \cite{richtershipley}.
\end{remark}

Notice that all left adjoints of the zig-zag of Quillen equivalences from Section \ref{sec:algModTsp} (see also \cite{BGKS} and \cite{tnqcore}) up to $H\bQ \leftmod$ are strong symmetric monoidal.

By Lemmas \ref{lem:liftQE}, \ref{lem:generalised_diag_lift}, \ref{lem:R-mod_lift} and Remark \ref{rem:cellularisation_lift} we obtain lifted model structures on the categories of $\Eoneinfty$-algebras and Quillen equivalences at the levels of $\Eoneinfty$-algebras as below. 
\[
\xymatrix@R=1pc @C=7pc{
 & \algin{\Eoneinfty}{\TSP_\bQ}
\ar@<-1ex>[dd]_{\bbSfork \smashprod - } & **[l] \mathrm{in\ } \TSP
\\ \  \\
 & \algin{\Eoneinfty}{K_{\Top} \cell \bbSfork_{\Top} \leftmod}
\ar@<-1ex>[uu]_{\mathrm{pb}}
\ar@<+1ex>[dd]^{(-)^{\torus}} & **[l] \mathrm{in\ } \TSP
\\ \  \\
& \algin{\Eoneinfty}{K_{\Top}^{\torus} \cell (\bbSfork_{\Top})^\bT \leftmod}
\ar@<+1ex>[uu]^{\zeta_{\#}} & **[l] \mathrm{in\ } \Sp
}
\]

At this point we change the operad from $\Eoneinfty$ to $\Comm$ obtaining a Quillen equivalence.
\begin{lemma}\label{lem:einftocomm}\cite[Lemma 6.2]{BGKfinite}
There is an adjunction
\[
\adjunction{\eta_*}
{\algin{\Eoneinfty}{   \spO_\bQ   }}
{\algin{\Comm}{   \spO_\bQ   }}
{\eta^*}
\]
induced by the map of operads in $\Top_*$ $\eta: \Eoneinfty \lra \Comm$. This adjunction is a Quillen equivalence with respect to right-induced model structures from the rational positive stable model structure on $\spO_\bQ$.
\end{lemma}

\begin{lemma}There is an adjunction
\[
\adjunction{\eta_*}
{\algin{\Eoneinfty}{K_{\Top}^{\torus} \cell (\bbSfork_{\Top})^\bT \leftmod}}
{\algin{\Comm}{K_{\Top}^{\torus} \cell (\bbSfork_{\Top})^\bT \leftmod}}
{\eta^*}
\]
induced by the map of operads in $\Top_*$, $\eta: \Eoneinfty \lra \Comm$. This adjunction is a Quillen equivalence with respect to right-induced model structures from $K_{\Top}^{\torus} \cell (\bbSfork_{\Top})^\bT \leftmod$.
\end{lemma}

Another application of  Lemma \ref{lem:liftQE} 
we get two results.
\begin{lemma}
There is a Quillen equivalence
\[
\adjunction{\mathbb{P} \circ |-|}
{\algin{\Comm}{   \spS_\bQ    }}
{\algin{\Comm}{   \mathrm{Sp}_\bQ  }}
{\mathrm{Sing}\circ \mathbb{U}}
\]
where both model structures are right induced from the rational positive stable model structures on $\spS$ and $\mathrm{Sp}$ respectively.
\end{lemma}

\begin{proposition}
There is a Quillen equivalence
\[
\adjunction{H\bQ\wedge -}
{\algin{\Comm}{   \spS_\bQ   }}
{\algin{\Comm}{   (H\bQ \leftmod) }}
{U}.
\]
\end{proposition}

These results lift to the level of diagrams of modules over a fork of commutative rings used in Section \ref{sec:algModTsp} and so does the following result.

\begin{theorem}[{\cite[Corollary 8.4]{richtershipley}}]
There is a zig-zag of Quillen equivalences between the model category of commutative
$H \bQ$--algebras (in symmetric spectra) and differential graded commutative $\bQ$--algebras.
\end{theorem}

The model structure on commutative $H \bQ$--algebras is lifted from the positive
flat stable model structure on symmetric spectra, see for example \cite{Shipley_convenient}. The model structure on
differential graded commutative $\bQ$--algebras has fibrations the surjections
and weak equivalences the homology isomorphisms.

We can summarise the above in the next lemma, where $\bbSfork_t$ is a diagram of commutative rings obtained from $(\bbSfork_{\Top})^\bT$ by applying the result of \cite{shiHZ}. 
\begin{lemma} The adjunction
\[
\xymatrix@R=1pc @C=7pc{
& \algin{\Eoneinfty}{K_{\Top}^{\torus} \cell (\bbSfork_{\Top})^{\bT} \leftmod}
\ar@<+0ex>[dd]^{\mathrm{of\ Quillen\ equivalences}} & **[l] \mathrm{in\ Sp}
\\ \  \\
& \algin{\Comm}{K_t \cell \bbSfork_t \leftdmod}
\ar@<+0ex>[uu]^{\mathrm{zig-zag}} & **[l] \mathrm{in\ Ch}(\bQ)
}
\]
is a Quillen equivalence.
\end{lemma}

The formality argument of certain commutative rings and 
the removal of cellularisation of Section \ref{sec:algModTsp}
passes to Quillen equivalences as below, via Lemma \ref{lem:liftQE}.
\[
\xymatrix@R=1pc @C=7pc{
& \algin{\Comm}{K_t \cell \bbSfork_t \leftdmod}
\ar@<+0ex>[dd]^{\mathrm{of\ Quillen\ equivalences}} & **[l] \mathrm{in\ Ch}(\bQ)
\\ \  \\
& \algin{\Comm}{K_a  \cell \bbSfork_a \leftdmod}
\ar@<+0ex>[uu]^{\mathrm{zig-zag}}
\ar@<+1ex>[dd]^{\Gamma} & **[l] \mathrm{in\ Ch}(\bQ)
\\ \ \\
& \algin{\Comm}{d \mcA(\mathbb{T})_{dual}}
\ar@<+1ex>[uu]^{\mu} & **[l] \mathrm{in\ Ch}(\bQ)}
\]

We collect all the Quillen equivalences from this section into one result.
\begin{theorem}\label{thm:fullcompare}
There is a zig-zag of Quillen equivalences
\[
\algin{\Eoneinfty}{  (\TSP_\bQ)   }
\ \simeq \
\algin{\Comm} d\mcA(\torus),
\]
where $ d\mcA(\torus)$ denotes the algebraic model for rational $\torus$-spectra from \cite{gre99, BGKS}.
\end{theorem}

We may extend the above results to modules over 
rational $\Eoneinfty$--ring $\bT$-spectra. 
Rather than give the somewhat complicated 
definition of a module over an algebra
over an operad (and the monoidal product of such modules), 
we use a different model category of spectra
and simply talk about modules over a commutative ring object. 

Recall a variant of the category of rational $\torus$--spectra: 
the category of unital $\mcL(1)$--modules in orthogonal $\bT$-spectra from \cite{BlumbergHillL1}, where $\mcL$ is the linear isometries operad for the universe $U=\bR^\infty$. Thus the additive universe is the complete $\bT$-universe and the multiplicative universe is taken to be $\bR^\infty$. 
This category will be denoted $\lspec$. 
Recall further that 
\[
\algin{\Comm}{\lspec}
\simeq
\algin{\Eoneinfty}{\TSP_\bQ},
\]
so that an $\Eoneinfty$--ring in rational orthogonal $\torus$--spectra has a strictly commutative in 
rational unital $\mcL(1)$--modules in orthogonal $\bT$-spectra, 
$\lspec$.

We use the fact that one may rewrite our zig-zag of Quillen equivalences between 
$\algin{\Eoneinfty}{\TSP_\bQ}$ and commutative rings in 
$d\mcA(\torus)$ to start from 
$\algin{\Comm}{\lspec}$
and use 
unital $\mcL(1)$--modules in orthogonal $\bT$-spectra
in place of $\TSP$, 
as discussed in Remark \ref{rmk:threechoices}.

\begin{corollary}\label{cor:modulesMQE} For any rational $\Eoneinfty$--ring $\bT$-spectrum $R$ there is a zig-zag of monoidal Quillen equivalences
\[
R\leftmodin{  (\lspec)  }
\ \simeq \
\widetilde{R}\leftmodin d\mcA(\torus),
\]
where $\widetilde{R}$ is the derived image of $R$ under the zig-zag of Quillen equivalences of Theorem \ref{thm:fullcompare}.
\end{corollary}
\begin{proof}
Taking the symmetric monoidal category of $\bar{R}$-modules to be the starting point, we may lift the zig-zag of Quillen equivalences described just before the current corollary 
to categories of modules using Blumberg and Hill \cite[Theorem 4.13]{BlumbergHillL1} 
and Schwede and Shipley \cite[Theorem 3.12]{SchwedeShipleyMon}.
All steps of the zig-zag are symmetric monoidal, since the Quillen equivalences were so. The algebraic model $d\mcA(\torus)$ is taken with the dualizable model structure from \cite{barnesmonoidalmodelso2} as in the current paper.
The end of the zig-zag is then a symmetric monoidal category $\widetilde{R}\leftmodin d\mcA(\torus)$, where  $\widetilde{R}$ can be taken to be the derived image of $R$ under the zig-zag of Quillen equivalences of Theorem \ref{thm:fullcompare}, and hence is a commutative ring in $d\mcA(\torus)$.
\end{proof}

\part{An application to equivariant elliptic cohomology}

The purpose of the second part is to apply the results of Part 1 to a 
case of particular interest. 

\section{An island of algebraic geometry in rational \texorpdfstring{$\T$}{T}-spectra}
\subsection{Overview}
As mentioned in the introduction, given an  elliptic curve $C$ over a
$\Q$-algebra $K$, together with some coordinate data,  one may
construct a $\T$-equivariant elliptic cohomology $EC$ associated to $C$
\cite{gre05}.  Using $C$ and the
coordinate data  one may write down an  object $EC_a$ of $d\cA (\T)$,
which by  \cite{gre99} corresponds to a  $\T$-spectrum $EC_{\Top}$,
and by \cite{BGKS} this may be taken to be a ring. The construction of $EC_a$ is recalled in  Subsection \ref{subsec:ECa}. Furthermore,  $d\cA (\T)$ is a symmetric 
monoidal category and $EC_a$ is visibly a commutative monoid. 

By Corollary \ref{cor:modulesMQE}, $EC_{\Top}$  
can be chosen to be an 
$\Eoneinfty$-algebra in rational $\T$-spectra, which for this part we take to be $\lspec$. This is the category of unital $\cL(1)$-modules in $\TSP_{\bQ}$, see Blumberg and Hill \cite{BlumbergHillL1}.
It therefore makes sense to discuss the monoidal model category of
$E_{\infty}^1$-modules over $EC_{\Top}$, and the following theorem states that it is
equivalent to a category of sheaves of quasi-coherent modules over
$C$. The exact definitions of the following categories will be given
as we proceed.

Since $EC$ is
2-periodic, we form a 2-periodic version $P\cOC$ of the structure
sheaf $\cOC$. The non-empty open sets of the usual Zariski topology on $C$ consist of the complements  of arbitrary finite sets of points, but we use the torsion point (tp) topology in which
only points of finite order may be deleted. The purpose of Part 2 is to prove the following result. 

\begin{theorem}
\label{cor:main}
There is a symmetric monoidal Quillen equivalence 
\[
EC_{\Top}\leftmodin\TSP_\bQ \simeq 
\textnormal{quasi-coherent} \endash P\cOC \leftdmodin 
\textnormal{sheaves}/C_{tp}
\]
giving an equivalence of tensor triangulated categories 
\[
\ho( EC_{\Top} \leftmodin \TSP_\bQ) \simeq  D(PC_{tp}), 
\]
where $\leftdmod$ denotes differential objects in graded modules. 
\end{theorem}

\begin{proof} We will prove the theorem by showing that both
  categories are formal and have the same abelian skeleton, $\cA (PC)$. More
  precisely, we   construct the following zig-zag of Quillen equivalences 
\begin{multline*}
EC_{\Top} \leftmodin \TSP_\bQ
\stackrel{(1)}\simeq 
EC_a \leftmodin d \cA (\T)
\stackrel{(2)}\simeq 
d\cA (PC)\\
\stackrel{(3)}\simeq 
d\cA (P\cOC) 
 \stackrel{(4)}\simeq 
\mbox{cell} \endash P\cOCfork \leftdmod /C_{tp}
\stackrel{(5)}\simeq 
P\cOC \leftdmod/C_{tp}
\end{multline*}

\begin{description}

\item[Equivalence (1)] is an application of the monoidal equivalence of Corollary  
\ref{cor:modulesMQE}. 

\item[Equivalence (2)] comes from the elementary reformulation of the abelian 
category stated in Lemma \ref{lem:ECamodcAPC}. 

\item[Equivalence (3)] arises by replacing skyscraper and constant sheaves by 
the associated modules as in Proposition \ref{lem:desheaf}

\item[Equivalence (4)] is the Cellular Skeleton Theorem \ref{lem:CSTS} in this context. 

\item[Equivalence (5)] is an immediate application of the Cellularization 
Principle (Lemma \ref{lem:PBfork}). 
\end{description}

\end{proof}

\begin{remark}
There are similar results for Lurie's elliptic 
cohomology associated to a derived elliptic curve
\cite{LurieElliptic1}. Lurie constructs the representing $\Eoneinfty$-ring spectra as sections of a structure 
sheaf constructed from a  derived elliptic curve. At the level of 
derived algebraic geometry the counterpart of our theorem is built 
into his machinery, and over $\Q$ classical consequences can be 
inferred by theorems such as \cite[Theorem 2.1.1]{LurieElliptic1}.

Unpublished work of Gepner and the second author
analyses the relationship between Lurie's theories and the theories $EC_{\Top}$.  
Lurie's theories are defined and natural for all compact
Lie groups,  in particular the coordinate data is determined by the
coordinate at the identity pullback along the power maps. Accordingly,
not all of the theories $EC_a$ are restrictions of such theories. 
\end{remark}

\section{Elliptic curves}
Let $C$ be an elliptic curve over  a $\Q$-algebra $K$. There are certain
associated structures that we recall. For more details on the theory of elliptic curves we refer the reader to \cite{silverman}.

\subsection{Functions, divisors and differentials}
We write $e$ for the identity of $C$ as an abelian group,  
 $C[n]=\ker (C\stackrel{n}\lra C)$ for the points of
order dividing $n$ and $\Cn$ for the points of exact order $n$. 
Coming to divisors, Abel's theorem states that $\Sigma_i n_i[x_i]$ is the divisor
of a regular function if and only if $\Sigma_i n_i=0$ and $\Sigma_i
n_i x_i=e$. 

We write $\cOC$ for the structure sheaf of the curve $C$, and $\cK$
for the (constant) sheaf of meromorphic functions with poles only at
points of finite order. We then write  $\cO_{C,x}$ for the ring of 
functions regular at $x$, and  $(\cOC)_x^{\wedge}$ for its completion. 
For a divisor $D$ we write 
$\cOC(D)=\{ f\in \cK \st \mathrm{div} (f)+D\geq 0\}$ for the sheaf of functions with
poles bounded by $D$.  

 By the Riemann-Roch theorem $H^0(C; \cOC(ne))$ is $n$-dimensional over $K$ if
$n>0$. It follows that there is a short exact sequence
$$0\lra C \lra \Pic (C)\lra \Z\lra 0$$
where $\Pic(C)$ is the Picard group of line bundles on $C$.
Thus any line bundle is determined up to equivalence by a point on $C$
and its degree. 
We are especially concerned with the line bundles associated to sums
of divisors $C[n]$ where we have in particular $\cOC (C[n])\cong \cOC
  (n^2e)$.

Since $C$ is an abelian group, the sheaf 
  $\OmegaC$ of differentials is a trivializable line bundle. 
We will make our theory 2-periodic by using $\OmegaC$ for the
periodicity. 

\begin{definition}
The periodic structure sheaf is the sheaf of graded rings $\PcOC=\bigoplus_{n \in \Z}(\OmegaC)^{\tensor n}$,
where $(\OmegaC)^{\tensor n}$ is the degree $2n$ part. A {\em periodic} $\cOC$-module $M$ is a graded module over $\PcOC$. 
\end{definition}

\subsection{Derived categories of sheaves of modules}
The structure sheaf $\cOC$ is a sheaf of rings over $C$ and we  consider complexes of 
 sheaves of quasi-coherent $\cOC$-modules. 

The non-empty open sets of the usual Zariski topology on $C$ consist
of the complements  of arbitrary finite sets of points. We use the
torsion point (tp) topology where the non-empty open sets are obtained  by
deleting a finite sets of points, each of which has finite order. 
In effect, we only permit poles at points of finite order.

Consider four classes of maps $f:\cP\lra \cQ$ of sheaves over
$C_{tp}$. 
\begin{itemize}
\item $\cF_1=\{ f\st f: \cP\stackrel{\cong}\lra \cQ\}$
\item $\cF_2=\{ f\st f(D): \cP(D)\stackrel{\cong}\lra \cQ(D) \mbox{ for all 
   effective torsion point divisors $D$} \}$
\item $\cF_3=\{ f\st f_x: \cP_x\stackrel{\cong}\lra \cQ_x \mbox{ for all 
points $x$ of finite order} \}$
\item $\gl=\{ f\st \Gamma (f): \Gamma(\cP)\stackrel{\cong}\lra \Gamma (\cQ) 
  \}$
\end{itemize}

The first three classes are in fact the same. 

\begin{lemma}\label{lem:monmod}
We have equalities $\cF_1=\cF_2=\cF_3$.  We use $\cF_{tp}$ to denote it 
and we call elements of this class \emph{tp-isomorphisms}. We call a map $f$ between complexes of sheaves a \emph{tp-equivalence} if $H_*(f)$ is a tp-isomorphism.
\end{lemma}

\begin{proof}
It is clear that $\cF_1\supseteq \cF_2$ since we can take the tensor
product with the locally free sheaf $\cOC(D)$. It is clear that
$\cF_2\supseteq \cF_3$ since $\cF_x$ is the direct limit of sections
of $\cF (D)$ over $D$ not containing $x$.

The proof that $\cF_3\supseteq \cF_1$ is the local to global property
of sheaves.   
\end{proof} 

The derived category $D(C_{tp})$  is formed from complexes of sheaves 
by inverting all homology tp-isomorphisms.  The method of Lemma \ref{lem:vanishing} below shows that not 
every global equivalence is a tp-equivalence of sheaves, but since
 $\cF \subseteq \gl$ we obtain a localization
$$D(C_{tp})\lra D(C_{tp})[\gl^{-1}]=: D_{\gl}(C). $$

\subsection{Model categories of sheaves of modules}

The natural way for a homotopy theorist to construct these derived categories
is as homotopy categories of model categories. As seen in
\cite{hov01sheaves},  there are numerous different model structures with the class
 of homology isomorphisms as weak equivalences. 

One method of constructing model structures is by specifying a
generating set  $\bbK$  of small objects. Writing
$S^n(M)$ for  the complex with $M$ in degree $n$ and zero elsewhere
and $D^{n+1}(M)$ for the mapping cone of the identity map of
$S^n(M)$ as usual we take generating cofibrations and acyclic cofibrations  to be
$$I(M)=\{ S^{n-1}(M)\lra D^n(M) \st n\in \Z, M\in \bbK\}$$
$$J(M)=\{ 0\lra D^n(M) \st n\in \Z, M\in \bbK\}. $$
We then attempt to use these to generate a model structure. 

We are particularly interested in monoidal model structures so we want
to take $\bbK$ to consist of 
flat objects. We define
$$\bbK_{tp}=\{ \cOC (D)\st D \mbox{ a divisor of torsion
    points}  \}.$$

\begin{proposition}
\label{prop:monmod}
The sets $I(\bbK_{tp})$ and $J(\bbK_{tp})$ generate a proper
monoidal model structure on the category of differential graded sheaves of quasi-coherent 
$\cOC$-modules, with tp-equivalences as weak equivalences. We call it the \emph{flat $tp$-model structure}.
\end{proposition}

We show that these sets generate a proper model structure follows from 
Quillen's argument as codified by  \cite[Theorem 1.7]{hov01sheaves}. 
The fact that it  is monoidal follows from flatness of objects in $\bbK_{tp}$. 
Before proceeding, we record the usual calculation of maps out of discs and spheres and categorical generators. 

\begin{lemma}
\label{lem:dnsn}
 For a complex $\cP$ of sheaves, 
$$\Hom (D^n(M), \cP)=\Hom (M, \cP_n)$$
and 
$$\Hom (S^n(M), \cP)=\Hom (M, Z_n(\cP)). \qed$$
\end{lemma}

\begin{lemma}\label{lem:genSetK}
The set $\bbK_{tp}$ is a set of generators for quasi-coherent $tp$-sheaves.
\end{lemma}

\begin{proof}
By definition the stalk of a sheaf $\cF$ at $x$ is the direct limit of
$\cF (U)$ over open sets $U$ containing $x$. The complement of such an open set $U$ is a closed set not containing $x$. Let $\cV_x$ denote the collection of closed sets not containing $x$ and $\mcD_x$ denote the collection of effective divisors not containing $x$. 

Since $\mcD_x$ is cofinal in $\cV_x$,  
$$\cF_x=\colim_{D\in \mcD_x}\cF(X\setminus D). $$
Finally
$$\cF(X\setminus D)=\colim_k \Gamma \cF (kD) \cong \colim_k\Hom (\cOC(-kD), \cF) . $$

The last step follows from 
the enriched isomorphism 
\[
\Hom (\cOC(-D), \cF)\cong \cF (D). \qedhere
\]
\end{proof}

\begin{remark}
The proof applies to an arbitrary quasi-projective variety with the
Zariski topology.
\end{remark}

\begin{proof}[Proof of Proposition \ref{prop:monmod}.]
As usual,  we define fibrations to be $\Jinj$ (i.e. the maps with the
 right lifting property with respect to all elements of $J$) and
cofibrations to be $\Icof$ (i.e. the maps with the left lifting
property with respect to all elements of $\Iinj$).

It remains to verify the conditions of \cite[Theorem
1.7]{hov01sheaves}. Firstly, $\bbK_{tp}$ is a generating
set by Lemma \ref{lem:genSetK} so we need to show that if $\cP$ is $\bbK_{tp}$-flabby then $\Hom (M, \cP)$
is acyclic for all $M\in \bbK_{tp}$. 

Suppose that $\cP$ is $\bbK_{tp}$-flabby. By Lemma \ref{lem:dnsn},
this means that the map 
$$\Gamma (\cP(D)_n)=\Hom(D^n(\cOC (-D)), \cP)\lra \Hom(S^{n-1}(\cOC
(-D)), \cP)
\lra \Gamma (Z_n\cP(D))$$
is surjective for all $D$. Taking colimits as in the lemma, we see
that
$$\cP_n \lra Z_n\cP$$
is an epimorphism on stalks, which is to say it is an epimorphism of
sheaves. Thus every cycle is a boundary and $\cP$ is
acyclic. Tensoring with the flat sheaf $\cOC(D)$ we see that
$\Hom (\cOC (-D), \cP)$ is acyclic as required. 

For the monoidal statement we need to check the pushout product
axiom. By \cite[Chapter 4]{hov99}, we need only check that $i\Box
j$ is an acyclic cofibration when  $i\in I$ and $j\in J$ and that
$i\Box i'$ is a cofibration for $i, i' \in I$. 

 If $i=(S^{m-1}(M)\lra
D^m(M))$ and $j=(0\lra D^n(N))$ this involves considering the map 
$$S^{m-1}(M)\tensor D^n(N)\lra D^m(M)\tensor D^n(N). $$
This is obtained from 
$$M\tensor N\lra D^1(M\tensor N)$$
by suspending and taking the mapping cones of the identity. In all our
examples $M\tensor N$ is in the class $\bbK_{\cC}$ so this is a
generating cofibration and the mapping cone of the identity will give
another cofibration with both terms acyclic. 

For $i=(S^{m-1}(M)\lra D^m(M)),  i'=(S^{m'-1}(M')\lra D^{m'}(M'))$ the
resulting map comes from the one for $
\overline{i}=(S^{m-1}(\Q)\lra D^m(\Q)),
\overline{i}'=(S^{m'-1}(\Q)\lra D^{m'}(\Q))$
in chain complexes over $\Q$ by tensoring with $M\tensor M'$. Since 
$\overline{i}\Box \overline{i}'$ is a cofibration in
$\Q$-modules  it follows that
$i\Box i'$ is a cofibration. 
\end{proof}

\begin{remark}
We will need the periodic variant, where we start from the sheaf
$\PcOC$  of {\em graded} rings.  We then consider
differential graded sheaves of quasi-coherent 
$\PcOC$-modules in the tp-topology, and form $D(\PcOC)$ as before. 
\end{remark}

\section{Hasse squares}

\subsection{The Hasse square for \texorpdfstring{$\torus$}{T}-spectra}

It will be useful for the reader to recall the model for rational $\bT$--spectra from Section \ref{sec:algModTsp}, because the models for quasi-coherent sheaves over an
elliptic curve are constructed analogously.

In particular, recall the Quillen equivalence between 
a cellularisation of 
$\Raforkmod$ and the algebraic model $d\cA (\T)$. 
We are replacing the cellularization by restricting to the torsion
subcategory, which we think of as the essential skeleton. 
Since we are preparing for the analogue, we recall the
  outline, which is an elaboration of \cite[Section 5]{gscell}. 
The essential fact is that we have an adjunction where
  the left adjoint is inclusion $i$  (clearly strong symmetric 
monoidal) and the right adjoint $\Gamma$ is the 
torsion functor. 

This is a Quillen pair, and we apply the
Cellularization Principle \cite[Theorem 2.1]{gscell} to get a Quillen equivalence. 
The main content of the proof is that a cellular equivalence between
torsion objects is a homology isomorphism. We recall the proof 
from \cite{tnqcore} partly because it is much simpler for the circle
and partly because we will need to adapt it to sheaves in Lemma
\ref{lem:CSTS} below.

\begin{lemma}
\label{lem:CSTT}
(Cellular Skeleton Theorem \cite[11.5]{tnqcore})
There is a monoidal Quillen equivalence 
\[
\jcell \Raforkmod \simeq d\cA (\T). 
\]
\end{lemma}

\begin{proof} 
We show that if $X=(N\lra \tF\tensor V)$ in $d\cA (\T)$ is cellularly trivial
then $X\simeq 0$.

First note 
$$0 \simeq \Hom (S^{-W}, X)\cong (\Sigma^WN\lra \tF \tensor V)$$
so that 
$$0 \simeq \colim_{W^{\T}=0} \Hom(S^{-W}, X)=(\cEi N\stackrel{\cong}\lra \tF\tensor
V)=e(V)$$
is trivial. Hence $V=0$,  $N=T$ is a torsion module and $X=f(T)$.
Since $T$ is torsion, $T\cong \bigoplus_n 
e_nT$. Now consider the algebraic counterpart of the basic cell
$\Sigma^{-1}\sigma_n=\T\sm_{\T[n]}e_nS^{-1}$, namely $(\Q_n \lra \tF\tensor 0)$, where $\T[n]$ are points of order $n$ and $\Q_n$ denotes $\Q$ at degree $0$ at the spot corresponding to $C_n$ in the $\cOcF$-module. 
Then we have  
$$0\simeq \Hom (\sigma_n, f(T))=\Hom_{\Q[c]} (\Q, e_nT)$$
Hence $e_nT\simeq 0$. Since this applies to all $n\geq 1$ we see
$T\simeq 0$ as required.  
\end{proof}

For comparison it is illuminating to identify the non-trivial objects $X=(\beta: N \lra 
\tF\tensor V)$ of $d\cA(\bT)$ so that $\piT_*(X)=0$. 
\begin{lemma}
\label{lem:vanishing}
If $T$ is an $\cE$-divisible $\cE$-torsion module with action map
$\mu: \cOcF\tensor T \lra T$ then the object 
$$A(T)=(NT\lra \tF\tensor T) \mbox{ with } NT=\ker(\tF\tensor T\stackrel{\mu}\lra T)$$
has $\pi^{\T}_*(A(T))=0$. If $X$ has $\pi^{\T}_*(X)=0$ then $X\simeq
A(T)$ where $T=\pi^{\T}_*(X\sm \Sigma E\cF_+)$. 
\end{lemma}

\begin{proof} Suppose that $X=(N\stackrel{\beta}\lra \tF\tensor V)$ has $\pi^{\T}_*(X)=0$.
First, we note that any $X$ splits as an even part $X_{ev}$ and an odd 
part $X_{od}$. Since $X$ is $\piT_*$-acyclic so are its even and odd 
parts. We may therefore assume without loss of generality that $X$ is 
even. Next, we see that $\beta $ must be injective. Otherwise, if 
$K=\ker (\beta)$ we have a cofibre sequence 
$$X\lra e(V)\lra f(T)\oplus \Sigma f(K), $$
where $T=(\tF\tensor V)/N$.

Since $\piT_*(e(V))=V$ is in even degrees, we must have 
$\piT_*(\Sigma f(K))$ in even degrees, but $K$ is a torsion module, 
so that if it is not zero it will have some odd degree homotopy.

Since $X$ has injective structure map $\beta$,  we have 
an injective  resolution 
$$0\lra X\lra e(V)\lra f(T)\lra 0, $$
which is $\cE$-torsion and divisible. Now $\piT_*(X)=0$ if and only if $V=\piT_*(e(V))\lra 
\piT_*(f(T))=T$ is an isomorphsism, so that $X\simeq A(T)$ as
required. 
\end{proof}

\subsection{The Hasse square for sheaves over \texorpdfstring{$C$}{C}}

Returning to the elliptic curve $C$, just as in Section \ref{sec:algModTsp} there is a pullback square
$$\xymatrix{\cOC \ar[r]\ar[d]  &\cK\ar[d]\\
\cOCFh
\ar[r] &\cK \tensor_{\cOC}\cOCFh
}$$
of sheaves of rings over $C$. Recall that $\cK$ denotes the constant sheaf of meromorphic functions with poles only at points of finite order. We note that there is a product splitting 
$$\cOCFh \cong  \prod_n\cOCnhat. $$

We write $\cOCforkmod$ for the category of differential modules over the diagram 
$$\cOCfork=\left(
\vcenter{\xymatrix{
&\cK\ar[d]\\
\cOCFh\ar[r] &\cK \tensor_{\cOC}\cOCFh 
}}\right)$$
of sheaves of rings.

Once again, we arrange this on a line for typographical reasons 
$$\cOCFh 
\stackrel{l}\lra \cEi \cOCFh \stackrel{\iota}\lla \cK. $$
We write  $\cOCforkmod$ for the category of all diagrams 
$$\cN \stackrel{k}\lra \cP \stackrel{m}\lla \cV$$ 
of differential modules over $\cOCfork$ where $k$ is a map of differential $\cOCFh$-modules and $m$ is a map of differential $\cK$-modules. The category  $\cA (\cOC)$ consists of diagrams 
where the maps induce isomorphisms $l_*N\cong P$ and $\iota_*V\cong 
P$. 

 Just as happened for modules over the sphere
spectrum, the model category of  differential graded $\cOC \leftmod$ with
$tp$-equivalences from Lemma \ref{lem:monmod} can
be recovered from differential modules over $\cOCfork$.

\begin{lemma}
\label{lem:PBfork}
There is a Quillen equivalence  
\[
\cOC \leftdmod \simeq 
\textnormal{cell} \endash \cOCfork \leftdmod,
\]
where the model structure on the left is the one from Proposition \ref{prop:monmod} and the one on the right is the cellularization (at the derived images of the generators $\cO (D)$ of differential graded $\cOC$-modules) of  the diagram injective model structure built from the $tp$-flat model structures on the categories of sheaves.
\end{lemma}

\begin{proof}
Extension of scalars from $\cOC$ to $\cOCfork$ is left adjoint to the
pullback functor. If we use  the $tp$-flat model structures on the categories
of sheaves and the diagram injective model structure, then this is a
Quillen pair. 

The generators $\cO (D)$ of differential $\cOC$-modules are small, and have small
images under the derived left adjoint. The fact that the $\cOC$ is the pullback of the diagram
$\cOCfork$ and the fact that $\cOC(D)$ is locally free shows that all generators are locally free. The result then follows from
the Cellularization Principle \cite{gscell}. 
\end{proof}

Similarly, we may obtain a graded abelian category if we replace
$\cOC$ by $\PcOC$, and then form $d\cA (\PcOC)$.

We give the  categories of differential graded modules over the
individual graded sheaves of rings in $\PcOCfork$ the
flat $tp$-model structures, which we have seen to be 
proper and monoidal.  We then give the category $\PcOCforkmod$ the 
diagram-injective model structure.

\begin{lemma} \label{lem:CSTS}
(Cellular Skeleton Theorem)
There is a monoidal Quillen equivalence 
\[
\jcell \PcOCforkmod \simeq d\cA (\PcOC)
\]
where the left hand side is cellularised at the derived images of the generators of the differential $\PcOC$-modules.
Both of these model categories are Quillen equivalent to differential $\PcOC$-modules with the $tp$-flat model structure and
so all have homotopy categories $D(PC)$. 
\end{lemma}

\begin{proof}
We begin with the adjunction 
$$\adjunction{i}{d\cA (\PcOC)}{\PcOCforkmod}{\Gamma}. $$
The left adjoint is inclusion $i$ and the right adjoint $\Gamma$ is the 
torsion functor constructed just like in the case of $\cA (\T)$, see \cite{BGKS}.  
The left adjoint  is clearly strong symmetric 
monoidal. 

Passing to differential graded objects, we obtain  a monoidal Quillen
pair. The objects $\cOC (D)$ with $D$ a torsion point divisor lie in 
$d\cA (\cOC)$, so the Cellularization  Principle \cite{gscell} shows that the
cellularizations are Quillen equivalent.

Finally, we observe that any  
$\cOC(D)$-cellular equivalence of torsion objects for all  torsion point divisors $D$ is a homology
isomorphism. The proof is the precise analogue of the one in $d\cA
(\T)$ (Lemma \ref{lem:CSTT}). We show that if $X=(N\lra \tF\tensor V)$ is $\cOC(D)$-trivial for
all torsion point divisors $D$ then $X\simeq 0$.

First note 
\[
0 \simeq \Hom (\cO (-D), X)\cong (N(D)\lra \tF \tensor V)
\]
so that 
\[
0 \simeq \colim_D \Hom(\cOC (-D), X)=(\colim_D N(D)\stackrel{\cong}\lra \tF\tensor
V)=e(V)
\]
is trivial. Hence $V=0$,  $N=T$ is torsion sheaf and $X=f(T)$. It follows that its
stalks at non-torsion points are trivial. Finally we use the short
exact sequence 
$$0\lra \cOC\lra \cOC(C\langle n\rangle )\lra i^{C\langle n\rangle }_*k\lra 0$$
where $ i^{C\langle n\rangle }: C\langle n\rangle \lra \T$ denotes the inclusion.
We see 
\[
0\simeq \Hom (f(i^{C\langle n\rangle }_*k), f(T))=\Hom_{C\langle n\rangle } (k, T|_{C\langle n\rangle})
\]
Hence $T|_{C\langle n\rangle }\simeq 0$. Since this covers all torsion points
$T\simeq 0$ as required. 
\end{proof}

\subsection{From algebra to algebraic geometry}
We now wish to observe that the category $\cA (\cOC)$ of diagram of sheaves
is equivalent to a  purely algebraic category $\cA (C)$ of modules 
over the diagram of rings. All results in this section apply equally well to periodic modules, but since the results are about abelian categories we will state them for the non-periodic version for brevity.

 We note that the diagram $\cOCfork$ consists of constant sheaves and
 skyscraper sheaves, so it is equivalent to the  diagram of
 rings 
$$\Gamma \cOCfork
=\left(
\vcenter{\xymatrix{
&\cK\ar[d]\\
\prod_n \Gamma \cO_{\Cn}^{\wedge} 
\ar[r] &\cK \tensor_{\cOC}\prod_n \Gamma \cO_{\Cn}^{\wedge} 
}}\right)$$
(in modules over the ground ring $K$) formed by taking global sections. 
 We let $\GcOCforkmod$ denote  the category of modules over $\Gamma \cOCfork$
 and we let $\cA (C)$ to be the category of these $\Gamma \cOCfork$-modules so that the horizontal and
 vertical maps are extensions of scalars. We again arrange this on a line for typographical reasons 
$$\prod_n \Gamma \cO_{\Cn}^{\wedge} \stackrel{l}\lra 
\cK \tensor_{\cOC}\prod_n \Gamma \cO_{\Cn}^{\wedge} 
\stackrel{\iota}\lla \cK. $$
We write  $\GcOCforkmod$ for the category of all diagrams 
$$\cN \lra \cP \lla \cV$$ 
of modules over $\Gamma \cOCfork$. The category  $\cA (C)$ consists of diagrams 
where the maps induce isomorphisms $l_*N\cong P$ and $\iota_*V\cong 
P$. 

\begin{lemma}
\label{lem:desheaf}
We have an equivalence of categories 
$$\mbox{$\cA (\cOC)$}\simeq \cA (C).  $$
\end{lemma}

\begin{proof}
For clarity,  in this proof we will introduce notation to distinguish between sheaves and
global sections. 

At the generic point we note that the sheaf $\widetilde{\cK}$ of meromorphic
functions is constant at $\cK=\Gamma \widetilde{\cK}$. Similarly any quasi-coherent
module $\widetilde{\cM}$ over the sheaf $\widetilde{\cK}$ is also constant 
(at $\cM=\Gamma \widetilde{\cM}$). This means that passage to
sections gives an equivalence between sheaves of modules over $\widetilde{\cK}$
and modules over the ring $\cK$.

Next, if $\cF$ is any skyscraper sheaf of rings concentrated over a
finite set of points $x_1, \ldots, x_s$ passage to global sections gives an
isomorphism 
$$\Gamma \cF\stackrel{\cong}\lra \prod_{i=1}^s \cF_{x_i}$$
between the global sections and the product of the stalks.
Furthermore,   any $\cF$-module $\cN$ is a skyscraper over $x_1,
\ldots , x_s$, because $0=1$ in the stalk at any other
point. Next, we have an isomorphism 
$$\Gamma \cN\stackrel{\cong}\lra \prod_{i=1}^s \cN_{x_i}. $$
It follows that there is an equivalence of categories between sheaves of 
$\cF$-modules and modules over $\prod_{i=1}^s \cF_{x_i}$. Finally,
this is natural for maps of sheaves of rings concentrated over $x_1, \ldots  , x_s$.

Next, if $\widetilde{\cC}$ is constant at $\cC$ and $\cN$ is concentrated at $x_1, \ldots,
x_s$ then passage to global sections gives an isomorphism
$$\Gamma : \Hom(\widetilde{\cC}, \cN)\stackrel{\cong}\lra
\prod_{i=1}^s\Hom(\cC, \cN_{x_i}). $$

Finally we turn to the  question of dealing with an infinite
number of points. In our case we have the sheaf, $\cOCFhat$ which is
itself a product: $\cOCFhat=\prod_n \cOCnhat$. 
Since $\Gamma$ is lax monoidal, passage to sections therefore gives a functor 
\[
\Gamma : \cOCFhat \leftmod=
\left[ \prod_n\cOCnhat \right] \leftmod
\lra 
\left[ \prod_n\Gamma \cOCnhat \right] \leftmod. 
\]

Altogether, taking sections at all three points of the diagram gives a functor 
\[
\mathbf{\Gamma}: \cOCfork \leftmod \lra \Gamma\cOCfork \leftmod. 
\]

Next we define a functor in the other direction
\[
\mathbf{t}: \Gamma \cOCfork \leftmod \lra \cOCfork \leftmod.
\]
On the two right hand entries,  we use the constant sheaf functor. 
For the bottom left, we note that if $M$ is a module over 
$\prod_n\Gamma\cOCnhat$ we may define a sheaf $t\cM$ as
follows. The open sets are obtained by choosing a finite set $n_1,
\ldots , n_t$ of orders and taking 
$$U(n_1, \ldots , n_t)=C\setminus \coprod_{j}C\langle
n_j\rangle. $$
Associated to the numbers $n_1, \ldots , n_t$ there is an idempotent
$e(n_1, \ldots , n_t)$ in $\prod_n\Gamma\cOCnhat$  supported on the subgroups of these orders and 
we take 
$$t(\cM)(U(n_1, \ldots, n_t))=\cM/e(n_1, \ldots , n_j)\cM. $$

The map $\cM \lra \cP$ along the horizontal map in diagram defining $\Gamma \cOCfork$ gives a map
$t(\cM)\lra \widetilde{\cP}$, where $\widetilde{\cP}$ is the
constant sheaf.  Since $\Gamma$ and $t$ are inverse equivalences, so
are $\mathbf{\Gamma}$ and $\mathbf{t}$.
\end{proof}

\begin{remark}It is clear that the analogous statement works for differential objects and also for $\PcOC$ and $PC$ instead of $\cOC$ and $C$ respectively. Thus we have an equivalence of categories 
$$\mbox{$d\cA (\PcOC)$}\simeq d\cA (PC).  $$
\end{remark}

\section{The algebraic model of \texorpdfstring{$\T$}{T}-equivariant elliptic cohomology}

\subsection{The algebraic object of an elliptic curve with coordinates}
\label{subsec:ECa}
Coordinate data on our elliptic curve $C$ consists of functions $t_1, t_2, \ldots
$ with $t_n$ vanishing to the first order on the points of exact order
$n$. We also let $Dt$ denote the regular differential agreeing with
$dt_1$ at $e$.

 This data allows us to write down an object $EC_a$ in the algebraic model for rational $SO(2)$-spectra. Indeed we
will define 
$$EC_a=(NC \stackrel{\beta}\lra \cEi \cOcF\tensor VC)$$
where $\beta$ is injective. We take 
$$VC=P\cK$$
to  consist of the meromorphic functions made periodic, and $NC$ to be the following
module of regular functions: 
$$NC=\ker \left[ \cOcF \tensor VC  \stackrel{p}\lra \bigoplus_{n\geq 1}H^1_{C\langle n
  \rangle}(C; \PcOC) \right]$$
where the principal part map $p$ has $n$th component 
$$p(c^v\tensor f)=\overline{  \left( \frac{t_n}{Dt}\right)^{v(n)}f
},   $$
where $v: \cF \lra \Z$ is a function which is zero almost everywhere (it is
checked in \cite[Lemma 10.8]{gre05} that this determines the map $p$). 

\begin{remark}
For $i=0, 1$ we have 
$$EC_{-i}^{\T}(S^V)=H^i(C;\cO (D(V)). $$
The complete proof is given in \cite[Section 10]{gre05}, but the $i=0$ part of the result
is immediate from the definition. Indeed,  $c^v\tensor f$ lies in the
kernel of $p$ if, for every point $Q$ of finite order $n$,  the pole
of $f$  at $Q$ is of order $\leq v(n)$ (i.e., $\nu_Q(f)+v(n)\geq 0$):
$$\ker (p|_{c^v\tensor \cK})=\Gamma (\cO (D(v))), $$
where 
$$D(v)=\sum_nv(n)C\langle n \rangle.$$ 
\end{remark}

\subsection{Multiplicativity of \texorpdfstring{$EC_a$}{ECa}}

The product of two meromorphic functions is meromorphic. 
Indeed if $f$ and $g$ have  poles at points only of finite order
the product $fg$ has the same property. Hence $\cK$ is a commutative
ring and $P\cK$ and $\cEi\cOcF\tensor P\cK$ are commutative graded
rings in even degrees. 

\begin{lemma}  (\cite[Theorem 11.1]{gre05})
Multiplication of meromorphic functions defines a 
 map $EC_a\tensor EC_a \lra EC_a$ making $EC_a$ into a commutative
 monoid in $d\cA (\T)$. 
\end{lemma}

\subsection{From modules over \texorpdfstring{$EC_a$}{ECa} to modules 
over \texorpdfstring{$\Gamma \cOCfork$}{C}. }
The connection between differential $EC_a$-modules and $d\cA (PC)$ is entirely
elementary.

\begin{lemma}
\label{lem:ECamodcAPC}
We have an equivalence of abelian categories 
$$EC_a\leftmodin\cA (\T)\simeq \cA (PC). $$
\end{lemma}

\begin{proof}
This is an immediate reformulation: we just need to make it apparent
how the same structures are expressed in the two categories. 

We have $EC_a=(NC\lra \tF\tensor P\cK)$ so that if $X=(N \lra \tF\tensor
V)$ is a module over $EC_a$ we have a map 
$$EC_a\tensor X=(NC\tensor_{\cOcF}N \lra \tF\tensor P\cK\tensor V)
\lra (N \lra \tF\tensor V) =X. $$
This makes $V$ into a $P\cK$-module, $N$ into an $NC$-module, and the
Euler classes $e(V)$
 come from $EC_a$ (locally $e(V)(n)=(t_n/Dt)^{d_n}$,
where $d_n=\dim_{\C}(V^{\T[n]})$). 
For definition and detailed discussion of Euler classes in case of a circle group see \cite{gre99}.
\end{proof}

This lemma clearly extends to the level of differential objects in both abelian categories and thus finishes the proof of Theorem \ref{cor:main}.

\bibliographystyle{alpha}
\bibliography{somebib}
\end{document}